\documentclass[article,12pt]{amsart}

\usepackage{amsfonts}
\usepackage{amsmath}
\usepackage{amssymb}
\usepackage{graphicx}
\usepackage{fullpage}
\usepackage{epsfig}
\usepackage{enumitem}
\usepackage{multirow}
\usepackage{color}

\allowdisplaybreaks

\numberwithin{equation}{section}

\newcommand{\oyy}{\Omega_{yy}}
\newcommand{\oyx}{\Omega_{yx}}
\newcommand{\oyxt}{\Omega_{yx}^{\, t}}
\newcommand{\oxx}{\Omega_{xx}}
\newcommand{\oyys}{\Omega_{yy}^{*}}
\newcommand{\oyxs}{\Omega_{yx}^{*}}
\newcommand{\oyxts}{\Omega_{yx}^{*\, t}}

\newcommand{\syy}{\Sigma_{yy}}
\newcommand{\syx}{\Sigma_{yx}}
\newcommand{\syxt}{\Sigma_{yx}^{\, t}}
\newcommand{\sxx}{\Sigma_{xx}}
\newcommand{\sxxi}{\Sigma_{xx}^{-1}}
\newcommand{\sxxs}{\Sigma_{xx}^{*}}
\newcommand{\syys}{\Sigma_{yy}^{*}}
\newcommand{\syxs}{\Sigma_{yx}^{*}}
\newcommand{\ts}{\theta^{\, *}}
\newcommand{\oyyi}{\oyy^{-1}}
\newcommand{\oyyis}{\Omega_{yy}^{*\, -1}}
\newcommand{\varx}{S^{\, (n)}_{xx}}
\newcommand{\vary}{S^{\, (n)}_{yy}}
\newcommand{\covyx}{S^{\, (n)}_{yx}}

\newcommand{\dyy}{\delta \theta_{yy}}
\newcommand{\dyx}{\delta \theta_{yx}}
\newcommand{\dyxt}{\delta \theta_{yx}^{\, t}}
\newcommand{\dt}{\delta \theta}
\newcommand{\dvt}{\delta \vartheta}
\newcommand{\dvyy}{\delta \vartheta_{yy}}
\newcommand{\dvyx}{\delta \vartheta_{yx}}
\newcommand{\tr}{\textnormal{tr}}
\newcommand{\hsp}{\hspace{0.5cm}}
\renewcommand{\geq}{\geqslant}
\renewcommand{\leq}{\leqslant}

\newcommand{\llangle}{\langle\!\langle}
\newcommand{\rrangle}{\rangle\!\rangle}
\renewcommand{\vec}{\textnormal{vec}}
\newcommand{\vech}{\textnormal{vech}}
\newcommand{\diag}{\textnormal{diag}}
\newcommand{\lmax}{\lambda_{\textnormal{max}}}
\newcommand{\lmin}{\lambda_{\textnormal{min}}}
\newcommand{\wh}{\widehat}
\newcommand{\wht}{\widehat{\theta}}
\newcommand{\whoyy}{\widehat{\Omega}_{yy}}
\newcommand{\whoyx}{\widehat{\Omega}_{yx}}
\newcommand{\omli}{\underline{\omega}_{L}}
\newcommand{\omls}{\overline{\omega}_{L}}
\newcommand{\omss}{\overline{\omega}_{S}}
\newcommand{\corr}{\dC\textnormal{orr}}
\newcommand{\cG}{\mathcal{G}}
\newcommand{\cN}{\mathcal{N}}
\newcommand{\dC}{\mathbb{C}}

\newcommand{\dR}{\mathbb{R}}
\newcommand{\dS}{\mathbb{S}}
\newcommand{\de}{\mathrm{e}}

\theoremstyle{plain}
\newtheorem{thm}{Theorem}[section]
\newtheorem{lem}{Lemma}[section]
\newtheorem{prop}{Proposition}[section]

\theoremstyle{remark}
\newtheorem{rem}{Remark}[section]

\begin{document}

\title[A partial graphical model with prior on the direct links]
{A partial graphical model with a structural prior on the direct links between predictors and responses
\vspace{2ex}}
\author[E. Okome Obiang]{Eunice Okome Obiang}
\address{Univ Angers, CNRS, LAREMA, SFR MATHSTIC, F-49000 Angers, France.}
\email{okome@math.univ-angers.fr}
\author[P. J\'ez\'equel]{Pascal J\'ez\'equel}
\address{1 Unit\'e de Bioinfomique, Institut de Canc\'erologie de l'Ouest, Bd Jacques Monod, 44805 Saint Herblain Cedex, France.\vspace{-2ex}}
\address{2 SIRIC ILIAD, Nantes, Angers, France. \vspace{-2ex}}
\address{3 CRCINA, INSERM, CNRS, Universit\'e de Nantes, Universit\'e d'Angers, Institut de Recherche en Sant\'e-Universit\'e de Nantes, 8 Quai Moncousu - BP 70721, 44007, Nantes Cedex 1, France.}
\email{pascal.jezequel@ico.unicancer.fr}
\author[F. Pro\"ia]{Fr\'ed\'eric Pro\"ia}
\address{Univ Angers, CNRS, LAREMA, SFR MATHSTIC, F-49000 Angers, France.}
\email{frederic.proia@univ-angers.fr}

\thanks{}

\begin{abstract}
This paper is devoted to the estimation of a partial graphical model with a structural Bayesian penalization. Precisely, we are interested in the linear regression setting where the estimation is made through the direct links between potentially high-dimensional predictors and multiple responses, since it is known that Gaussian graphical models enable to exhibit direct links only, whereas coefficients in linear regressions contain both direct and indirect relations (due \textit{e.g.} to strong correlations among the variables). A smooth penalty reflecting a generalized Gaussian Bayesian prior on the covariates is added, either enforcing patterns (like row structures) in the direct links or regulating the joint influence of predictors. We give a theoretical guarantee for our method, taking the form of an upper bound on the estimation error arising with high probability, provided that the model is suitably regularized. Empirical studies on synthetic data and a real dataset are conducted.
\end{abstract}

\maketitle

\vspace{-0.5cm}

\begin{center}
\end{center}

\medskip

\section{Introduction}
\label{SecIntro}

We are interested in the recovery and estimation of direct links between high-dimensional predictors and a set of responses. Whereas the graphical models seem a natural way to go, we propose to take account of a prior knowledge on the predictors, when possible. This is typically the case when dealing with genetic markers whose joint influence may be anticipated thanks to some kind of genetic distance, or when the predictors are supposed to represent a continuous phenomenon so that consecutive covariates probably act together. In this regard, while taking up the graphical approach, we introduce some Bayesian information in a structural regularization of the estimation procedure, although the inference remains frequentist, thereby following the idea of Chiquet \textit{et al.} \cite{ChiquetEtAl17}. This strategy also enables to affect the amount of shrinkage by playing with some hyperparametrization in the prior, while sparsity may be obtained \textit{via} usual penalty-based patterns. Regarding the mathematical formalization of the graphical models that we will just briefly discuss in this introduction, we refer the reader to the very complete handbook recently edited by Maathuis \textit{et al.} \cite{MaathuisEtAl18}. We also refer the reader to the book of Hastie \textit{et al.} \cite{HastieEtAl15} and to the one of Giraud \cite{Giraud14}, both related to the standard high-dimensional statistical methods. Before introducing the model and the organization of this work, let us describe the notation used throughout the paper.

\subsection{Notation}

For any matrix $A$, $\vert A \vert_{*} = \Vert \vec(A) \Vert_{*}$ is the elementwise $\ell_{*}$ norm of $A$ and $\vert A \vert_{*}^{-}$ is $\vert A \vert_{*}$ deprived of the diagonal terms of $A$. We also note $\Vert A \Vert_{F} = \vert A \vert_{2}$ the Frobenius norm of $A$ and $\Vert A \Vert_{2}$ the spectral norm of $A$. The Frobenius inner product between any matrices $A$ and $B$ of same dimensions is $\llangle A, B \rrangle = \langle \vec(A), \vec(B) \rangle = \tr(A^{t}\, B)$ whereas $\langle u, v \rangle = u^{t}\, v$ is the inner product of the Euclidean real space. For any vector $u$, $\vert u \vert_0$ is the number of non-zero values in $u$. For a matrix $A$, $[A]_{C}$ is to be understood as the matrix $A$ whose elements outside of the set of coordinates $C$ are set to zero and $\vec(A)$ is the vectorization of $A$ into a column vector. The eigenvalues of a square matrix $A$ of size $d$ with spectrum $\textnormal{sp}(A)$ are $\lambda_{i}(A)$ taken in decreasing order (from $\lambda_1(A) = \lmax(A)$ to $\lambda_{d}(A) = \lmin(A))$. The cones of symmetric positive semi-definite and definite matrices of dimension $d$ are $\dS_{+}^{\, d}$ and $\dS_{++}^{\, d}$ respectively.

\subsection{The partial graphical model}

In the classic Gaussian graphical model (GGM) setting, we aim at estimating the precision matrix $\Omega = \Sigma^{-1}$ of jointly normally distributed random vectors $Y \in \dR^{q}$ and $X \in \dR^{p}$ with zero mean and covariance $\Sigma$. The point is that it induces a graphical structure among the variables and the support of $\Omega$ is closely related to the conditional interdependences between them. Let us consider, now and in all the study, the sample covariances of $n$ independent observations $(Y_{i}, X_{i})$, denoted by
\begin{equation}
\label{SampleCov}
\vary = \frac{1}{n} \sum_{i=1}^{n} Y_{i}\, Y_{i}^{t}, \hsp \covyx = \frac{1}{n} \sum_{i=1}^{n} Y_{i}\, X_{i}^{t} \hsp \text{and} \hsp \varx = \frac{1}{n} \sum_{i=1}^{n} X_{i}\, X_{i}^{t}.
\end{equation}
Maximizing the penalized likelihood of a GGM boils down to finding $\Omega \in \dS_{++}^{\, p+q}$ that minimizes the convex objective
\begin{equation}
\label{LikGGM}
L_{n}(\Omega) = -\ln \det(\Omega) + \llangle S^{\, (n)}, \Omega \rrangle + \lambda\, \textnormal{pen}(\Omega)
\end{equation}
where $S^{\, (n)}$ is the full sample covariance built from the blocks \eqref{SampleCov}. The penalty function $\textnormal{pen}(\Omega)$ is usually $\vert \Omega \vert_1$ or even $\vert \Omega \vert_1^{-}$. Efficient algorithms exist to get solutions for \eqref{LikGGM}, see \textit{e.g.} Banerjee \textit{et al.} \cite{BanerjeeEtAl08}, Yuan and Lin \cite{YuanLin07}, Lu \cite{Lu09} or the graphical Lasso of Friedman \textit{et al.} \cite{FriedmanEtAl08}. The reader may also look at the theoretical guarantees of Ravikumar \textit{et al.} \cite{RavikumarEtAl11}. However, thinking at $X_{i}$ as a predictor of size $p$ associated with a response $Y_{i}$ of size $q$, the partial Gaussian graphical model (PGGM), developped \textit{e.g.} by Sohn and Kim \cite{SohnKim12} or Yuan and Zhang \cite{YuanZhang14}, appears as a powerful tool to exhibit direct relationships between the predictors and the responses. To understand this, consider the decomposition into blocks
\begin{equation*}
\Omega = \begin{pmatrix}
\oyy & \oyx \\
\oyxt & \oxx
\end{pmatrix} \hsp \text{and} \hsp \Sigma = \begin{pmatrix}
\syy & \syx \\
\syxt & \sxx
\end{pmatrix}
\end{equation*}
where $\oyy \in \dS_{++}^{\, q}$, $\oyx \in \dR^{q \times p}$ and $\oxx \in \dS_{++}^{\, p}$ and where the same goes for $\sxx$, $\syx$ and $\sxx$. The precision matrix $\Omega = \Sigma^{-1}$ satisfies, by blockwise inversion,
\begin{equation}
\label{InvBloc}
\oyyi = \syy - \syx\, \sxxi\, \syxt \hsp \text{and} \hsp \oyx = -(\syy - \syx\, \sxxi\, \syxt)^{-1}\, \syx\, \sxxi.
\end{equation}
The conditional distribution peculiar to Gaussian vectors
\begin{equation*}
Y_{i}\, \vert\, X_{i}\, \sim\, \cN(-\oyyi\, \oyx\, X_{i},\, \oyyi)
\end{equation*}
gives a new light on the multiple-output regression $Y_{i} = B^{\, t}\, X_{i} + E_{i}$ with Gaussian noise $E_{i} \sim \cN(0, R)$, through the reparametrization $B = -\oyxt\, \oyyi$ and $R = \oyyi$. Whereas $B$ contains direct and indirect links between the predictors and the responses (due \textit{e.g.} to strong correlations among the variables), $\oyx$ only contains direct links, as it is shown by the graphical models theory. In other words, the direct links are closely related to the concept of partial correlations between $X$ and $Y$ (see Meinshausen and B\"ulmann \cite{MeinshausenBuhlmann06} or Peng \textit{et al.} \cite{PengEtAl09}, for the univariate case). For example, the direct link between predictor $k$ and response $\ell$ may be evaluated through the partial correlation $\corr(Y_{\ell}, X_{k}\, \vert\, Y_{\neq\, \ell}, X_{\neq\, k})$ contained, apart from a multiplicative coefficient, in the $\ell$-th row and $k$-th column of $\oyx$ (see \textit{e.g.} Cor. A.6 in \cite{Giraud14}) with the particularly interesting consequence that the support of $\oyx$ is sufficient to identify direct relationships between $X$ and $Y$. Hence, in the partial setting, the objective reduces to the estimation of the direct links $\oyx$ together with the conditional precision matrix of the responses $\oyy$. Maximizing the penalized conditional log-likelihood of the model now comes down to minimizing the new convex objective
\begin{eqnarray}
\label{LikPGGM}
L_{n}(\oyy, \oyx) ~ = ~ -\ln \det(\oyy) + \llangle \vary, \oyy \rrangle + 2\, \llangle \covyx, \oyx \rrangle \hsp \hsp \hsp \hsp \hsp \hsp \nonumber \\
\hsp \hsp \hsp \hsp \hsp \hsp +~ \llangle \varx, \oyxt\, \oyyi\, \oyx \rrangle + \lambda\, \textnormal{pen}(\oyy) + \mu\, \textnormal{pen}(\oyx)
\end{eqnarray}
over $(\oyy, \oyx) \in \dS_{++}^{\, q} \times \dR^{q \times p}$ for some usual penalty functions. It is worth noting that $\textnormal{pen}(\oyx)$ often plays a crucial role in modern statistics dealing with high-dimensional predictors (and the natural choice is $\vert \oyx \vert_1$ to get sparsity) while we may choose $\lambda=0$, for the responses are hardly numerous. In the seminal papers \cite{SohnKim12} and \cite{YuanZhang14}, the authors consider $\vert \oyy \vert_1$ and $\vert \oyy \vert_1^{-}$ for $\textnormal{pen}(\oyy)$, respectively. Yuan and Zhang \cite{YuanZhang14} also point out that no estimation of $\oxx$ is needed anymore. In a graphical model, the estimation of $\oyx$ and $\oyy$ depends on the accuracy of the estimation of $\Omega$ which, in turn, is strongly affected by the one of $\oxx$, especially in a high-dimensional setting. The partial model overrides this issue, the focus is on $\oyx$ and $\oyy$ while $\oxx$ has disappeared from the objective function \eqref{LikPGGM}. The latter is obtained either by considering the multiple-output Gaussian regression scheme, or, as it is done in \cite{YuanZhang14}, by eliminating $\oxx$ thanks to a first optimization step in \eqref{LikGGM}. In this paper, we will consider the penalties
\begin{equation}
\label{Pen}
\textnormal{pen}(\oyy) = \vert \oyy \vert_1^{-} \hsp \text{and} \hsp \textnormal{pen}(\oyx) = \vert \oyx \vert_1
\end{equation}
which correspond to the PGGM (Gm) of \cite{YuanZhang14}. The Spring (Spr) of \cite{ChiquetEtAl17} can also be seen as a PGGM but with no penalty on $\oyy$ (replaced with an additional structuring one on $\oyx$, we will come back to this point thereafter), so for Spr we may consider $\lambda=0$. The generalized procedure (GenGm) at the heart of the study relies on a combination between these two approaches. We will see in due time that we keep both the penalties of Gm and the structuring one of Spr on $\oyx$. Finally, the intermediate solution consisting in estimating $\oyy$ and $B$ through the conditional distribution $Y_{i}\, \vert\, X_{i} \sim \cN(B^{\, t}\, X_{i},\, \oyyi)$ with penalizations both on $B$ and $\oyy$, presented and analyzed by Rothman \textit{et al.} \cite{RothmanEtAl10} and by Lee and Liu \cite{LeeLiu12}, is better known as a multivariate regression with covariance estimation (MRCE). However, it has been shown that the objective function suffers from a lack of convexity and that the optimization procedure may be debatable, in addition to the less convenient setup for statistical interpretation ($B$ contains both direct and indirect influences) compared to PGGM. Without  claiming  to  be  exhaustive, let us conclude this quick introduction by citing some related works, like the structural generalization of the Elastic-Net of Slawski \textit{et al.} \cite{SlawskiEtAl10}, the Dantzig approach of Cai \textit{et al.} \cite{CaiEtAl11} put in practice on genomic data \cite{CaiEtAl13}, the greedy research of the non-zero pattern in $\Omega$ of Johnson \textit{et al.} \cite{JohnsonEtAl12}, the approach of Fan \textit{et al.} \cite{FanEtAl09} using a non-convex SCAD penalty to reduce the bias of the Lasso in the estimation of $\Omega$, the eQTL data analysis of Yin and Li \cite{YinLi11} which makes use of a sparse conditional GGM, and so on. All the references inside will complete this concise list.

\subsection{Organization of the paper}

To sum up, we have two goals in this paper :
\begin{enumerate}
\item Give some theoretical guarantees to the (slightly modified) model introduced in Chiquet \textit{et al.} \cite{ChiquetEtAl17}.
\item Generalize the result of Yuan and Zhang \cite{YuanZhang14} to the case where a structural penalization is added in the estimation step.
\end{enumerate}
In Section \ref{SecMod}, we introduce the model, consisting in putting a generalized Gaussian prior on the direct links before the procedure of estimation of $\oyy$ and $\oyx$, and we detail the new convex objective. Then we provide some error bounds for our estimates, useful as theoretical guarantees of performance. Section \ref{SecEmp} is devoted to empirical considerations. We explain how we deal with the minimization of the new objective and we test the method on simulations first, and next on a real dataset (a Canadian average annual weather cycle, see \textit{e.g.} \cite{RamsaySilverman06}). After a short conclusion in Section \ref{SecConclu}, we finally prove our results in Section \ref{SecPro}. The numerous constants appearing in the results and the proofs are gathered in the Appendix, for the sake of readability.

\section{A generalized Gaussian prior on the direct links}
\label{SecMod}

We use the definition given in formulas (1)-(2) of \cite{PascalEtAl13} for the so-called $d$-dimensional multivariate generalized Gaussian $\cG\cN(0, 1, V, \beta)$ distribution with mean 0, scale 1, scatter parameter $V \in \dS_{++}^{\, d}$ and shape parameter $\beta > 0$. According to the authors, the density takes the form of
\begin{equation*}
\forall\, z \in \dR^{d}, \hsp f_{V,\, \beta}(z) = \frac{\beta\, \Gamma(\frac{d}{2})}{\pi^{\frac{d}{2}}\, \Gamma(\frac{d}{2 \beta})\, 2^{\frac{d}{2 \beta}} \sqrt{\det(V)}}\, \exp\bigg(\!-\frac{ \langle z, V^{-1} z \rangle^{\beta}}{2} \bigg)
\end{equation*}
where $\Gamma$ is the Euler Gamma function.

\begin{figure}[h!]
\centering
\includegraphics[width=9cm]{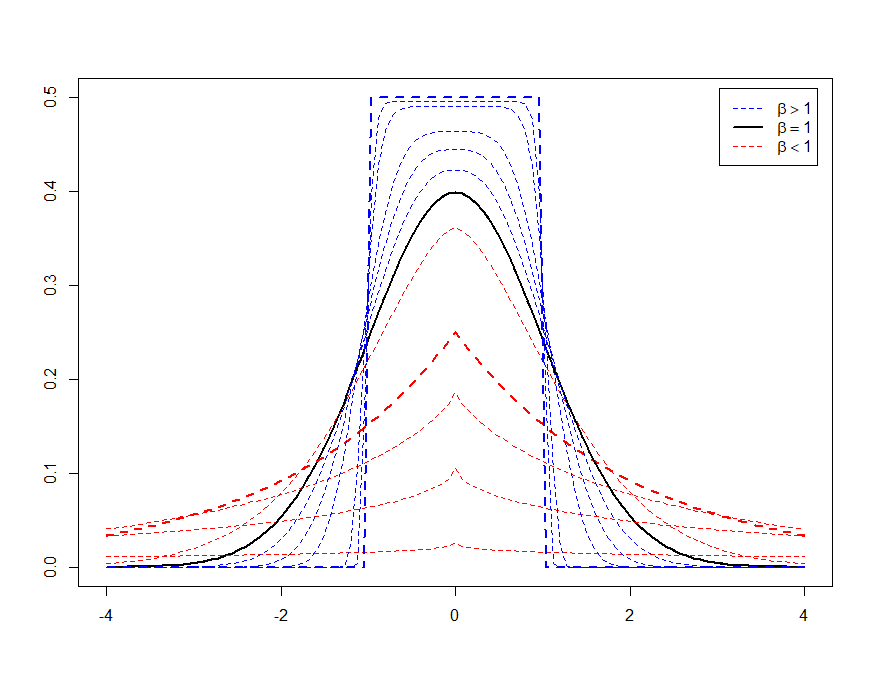}
\caption{Marginal shape of the generalized Gaussian distribution ($d=1$ and $V=1$) for some $\beta < 1$ (dotted red), $\beta=1$ (black) and some $\beta > 1$ (dotted blue). The noteworthy cases $\beta = 1/2$ (Laplace), $\beta=1$ (Gaussian) and $\beta = +\infty$ (uniform) are highlighted.}
\label{FigGGen}
\end{figure}

We clearly recognize the Gaussian $\cN(0, V)$ setting for $\beta=1$. Moreover, for $\beta=1/2$, it can be seen as a multivariate Laplace distribution whereas it is known to converge to some uniform distribution as $\beta \rightarrow +\infty$. The marginal shapes ($d=1$ and $V=1$) of the distribution are represented on Figure \ref{FigGGen}, depending on whether $\beta < 1$, $\beta=1$ or $\beta > 1$. Our results hold for all $\beta \geq 1$ but, as will be explained in due course, we shall not theoretically deviate too much from the Gaussianity in the prior (even if we will allow ourselves some exceptions in the practical works). The usual Bayesian approach for multiple-output Gaussian regression having $B$ as matrix of coefficients and $R$ as noise variance consists in a conjugate prior $\vec(B) \sim \cN(b, R \otimes L^{-1})$ for some information matrix $L \in \dS_{++}^{\, p}$ and a centering value $b$ (see \textit{e.g.} Sec. 2.8.5 of \cite{RossiEtAl12}). In the PGGM reformulation, we have $R = \oyyi$ and $B = -\oyxt\, \oyyi$ as explained in Section \ref{SecIntro}, and of course we shall choose $b=0$ to meet our purposes. Thus, 
\begin{equation*}
\vec(\oyxt) = -(\oyy \otimes I_{p})\, \vec(B)\, \sim\, \cN(0, \oyy \otimes L^{-1})
\end{equation*}
is a natural prior for the direct links (this is in particular the choice of the authors of \cite{ChiquetEtAl17}). Following the same logic, let us choose $\oyy \otimes L^{-1}$ for scatter parameter and suppose that
\begin{equation}
\label{Prior}
\vec(\oyxt)\, \sim\, \cG\cN(0, 1, \oyy \otimes L^{-1}, \beta).
\end{equation}
In this way, we can play on the intensity of the constraint we want to bring on $\oyx$, from a non-informative prior to quasi-boundedness through Laplace and Gaussian distributions. This prior entails an additional smooth term acting as a structural penalization in the objective \eqref{LikPGGM} that becomes
\begin{eqnarray}
\label{LikHPGGM}
L_{n}(\oyy, \oyx) ~ = ~ -\ln \det(\oyy) + \llangle \vary, \oyy \rrangle + 2\, \llangle \covyx, \oyx \rrangle \hsp \hsp \hsp \hsp \hsp \hsp \nonumber \\
\hsp \hsp \hsp \hsp \hsp \hsp +~ \llangle \varx, \oyxt\, \oyyi\, \oyx \rrangle + \eta\, \llangle L, \oyx^{\, t}\, \oyyi\, \oyx \rrangle^{\beta} + \lambda\, \vert \oyy \vert_1^{-} + \mu\, \vert \oyx \vert_1
\end{eqnarray}
with three regularization parameters $(\lambda, \mu, \eta)$. The smooth penalization lends weight to the prior on $\oyx$ and thereby plays on the extent of shrinkage and structuring through $\beta$, whereas $\vert \oyx \vert_1$ and $\vert \oyy \vert_1^{-}$ are designed to induce sparsity. One can note that this is closely related to the log-likelihood of a hierarchical model of the form
\begin{equation*}
\left\{
\begin{array}{l}
Y_{i}\, \vert\, X_{i}, \oyx\, \sim\, \cN(-\oyyi\, \oyx\, X_{i},\, \oyyi) \\ 
\vec(\oyxt)\, \sim\, \cG\cN(0, 1, \oyy \otimes L^{-1}, \beta)
\end{array}
\right.
\end{equation*}
where the emphasis is on $\oyx$ in the prior and $\oyy$ remains a fixed parameter, although it is important to see that, in this work, the estimation step does not rely on a posterior distribution. The following proposition is related to the existence of a global minimum for our objective \eqref{LikHPGGM} with respect to $(\oyy, \oyx)$ as soon as $\beta \geq 1$.
\begin{prop}
\label{PropConv}
Assume that $\beta \geq 1$. Then, $L_{n}(\oyy, \oyx)$ defined in \eqref{LikHPGGM} is jointly convex with respect to $(\oyy, \oyx)$.
\end{prop}
\begin{proof}
See Section \ref{SecProConv}.
\end{proof}

Now and throughout the rest of the paper, denote by $\theta = (\oyy, \oyx) \in \Theta = \dS_{++}^{\, q} \times \dR^{q \times p}$ the $(q \times (q+p))$-matrix of parameters of the model, with true value $\ts = (\oyys, \oyxs)$. As it is usually done in studies implying sparsity, we will also consider $S$ of cardinality $\vert S \vert$, the true active set of $\ts$ defined as $S = \{ (i,j),\, \ts_{i,j} \neq 0 \}$, and its complement $\bar{S}$. Our results also depends on some basic assumptions related to the true covariances of the Gaussian observations, and we will assume that the following holds.
\begin{equation}
\label{H1}\tag{H$_1$}
\sxxs \in \dS^{\, p}_{++}, \hsp \oyys \in \dS^{\, q}_{++}, \hsp B \neq 0~ \text{(that is, $\oyxs \neq 0$)} \hsp \text{and} \hsp \oyxs\, L\, \oyxts \in \dS_{++}^{\, q}.
\end{equation}
This is a natural hypothesis in our framework, in particular we suppose that there is at least a link between $X$ and $Y$.

\begin{rem}[Null model]
\label{RemNullMod}
Even if it is of less interest, our study does not exclude the case where $\oyxs=0$. Indeed, we might as well consider that $\oyxs=0$ and get the same results, but some constants should be refined. On the other hand, $\sxxs \in \dS^{\, p}_{++}$ and $\oyys \in \dS^{\, q}_{++}$ are crucial.
\end{rem}

Under \eqref{H1}, the random matrices
\begin{equation}
\label{MatA}
A_{n} = (\vary - \syys) - \oyyis\, \oyxs\, (\varx - \sxxs)\, \oyxts\, \oyyis \hsp \text{with} \hsp h_{a} = \vert A_{n} \vert_{\infty}
\end{equation}
and
\begin{equation}
\label{MatB}
B_{n} = 2\, ((\covyx - \syxs) + \oyyis\, \oyxs\, (\varx - \sxxs)) \hsp \text{with} \hsp h_{b} = \vert B_{n} \vert_{\infty}
\end{equation}
are going to play a fundamental role, especially  $h_{a}$ and $h_{b}$. Let us now provide some theoretical guarantees for the estimation of $\theta$ in our model, provided that the regularization parameters are located in a particular area $(\lambda, \mu, \eta) \in \Lambda$. Consider the penalized likelihood $\ell_{\lambda, \mu, \eta}(\theta)$ given in \eqref{LikHPGGM}, and estimate $\theta$ by the global minimum
\begin{equation}
\label{DefEst}
\wht = \arg \min_{\Theta}\, \ell_{\lambda, \mu, \eta}(\theta)
\end{equation}
obtained for $\beta \geq 1$. To facilitate reading, we postpone the precise definition of the numerous constants to the Appendix. We recall that $p$ is the number of predictors, $q$ is the number of responses and $\vert S \vert$ is the size of the true active set.
\begin{thm}
\label{ThmUppBound}
Fix $d_{\lambda} > c_{\lambda} > 1$, $d_{\mu} > c_{\mu} > 1$, $e_{\lambda} > 0$ and $e_{\mu} > 0$, and assume that the regularization parameters satisfy $(\lambda, \mu, \eta) \in \Lambda = [c_{\lambda}\, h_{a},\, d_{\lambda}\, h_{a}] \times [c_{\mu}\, h_{b},\, d_{\mu}\, h_{b}] \times [0,\, \overline{\eta}]$, where
\begin{equation*}
\overline{\eta} = \frac{\min\left\{ \frac{(c_{\lambda}-1)\, \lambda}{c_{\lambda}\, \ell_{a}},\, \frac{(c_{\mu}-1)\, \mu}{c_{\mu}\, \ell_{b}},\, \frac{e_{\lambda}\, h_{a}}{\ell_{a}},\, \frac{e_{\mu}\, h_{b}}{\ell_{b}} \right\}}{\beta\, s_{L}^{\, \beta-1}}
\end{equation*}
for some non-random constants $s_{L}$, $\ell_{a}$ and $\ell_{b}$ defined in \eqref{CstS} and \eqref{CstL}, and the random constants $h_{a}$ and $h_{b}$ given above. Then, under \eqref{H1}, there exists absolute constants $b_1 > 0$ and $b_2 > 0$ such that, for any $0 < b_3 < 1$ and as soon as $n > n_0$, with probability no less that $1 - \de^{-b_2 n}-b_3$, the estimator \eqref{DefEst} satisfies
\begin{equation*}
\Vert \wht - \ts \Vert_{F}\, \leq\, \frac{16\, m^{*}\, c_{\lambda, \mu} \sqrt{\vert S \vert}}{\gamma_{r,\eta,\beta,p}}\, \sqrt{\frac{\ln(10 (p+q)^2) - \ln(b_3)}{n}}
\end{equation*}
where $\gamma_{r,\eta,\beta,p}$, $c_{\lambda, \mu}$ and $m^{*}$ are technical constants defined in \eqref{CstGam}, \eqref{CstCb} and \eqref{CstMs}, respectively, and where the minimal number of observations is given by
\begin{eqnarray}
\label{NbMinObs}
n_0 & = & \max\Bigg\{ \frac{(\ln(10 (p+q)^2) - \ln(b_3))\, c_{\lambda, \mu}^2\, \vert S \vert\, (16\, m^{*})^2}{r^{*\, 2}\, \gamma_{r,\eta,\beta,p}^2}, \nonumber \\
 & & \hsp \hsp \hsp \hsp \hsp b_1\, (q + \lceil s_{\alpha} \rceil \ln(p+q)), \ln(10 (p+q)^2) - \ln(b_3) \Bigg\}
\end{eqnarray}
with $s_{\alpha}$ defined in \eqref{CstSa} and $r^{*}$ in \eqref{CstR}.
\end{thm}
\begin{proof}
See Section \ref{SecProThm}.
\end{proof}

Among all these constants, we can note that $s_{L}$, $\ell_{a}$, $\ell_{b}$, $h_{a}$ and $h_{b}$ are useful to properly describe and restrict $\Lambda$, the domain of validity of $(\lambda, \mu, \eta)$ for the theorem to hold. Once $\Lambda$ is fixed, the other constants take part in the upper bound of the estimation error. However, as it stands, the theorem is very difficult to interpret. The next two remarks seem essential to have an overview of the orders of magnitude involved for the number of observations, for $p$ and $q$, for the estimation error and for the regularization parameters. 

\begin{rem}[Validity band]
\label{RemVal}
Of course the degree of sparsity $\vert S \vert$ is crucial in the estimation error, but it also plays an indirect role in the probability associated with the theorem and in the numerous constants. In virtue of Lemma \ref{LemBorneFinale}, we can hope that $\lambda$ and $\mu$ have a wide validity band, by playing on $c_{\lambda}$, $c_{\mu}$, $d_{\lambda}$ and $d_{\mu}$. In turn, $\eta$ also has a non-negligible area of validity, provided of course that $\ell_{a}$, $\ell_{b}$ and $s_{L}$, all depending on combinations between $\oyxs$, $\oyyis$ and $L$, are small enough. Accordingly, it would be to our advantage if $L$ was both sparse and not chosen with too large elements. As it always appears together with $\eta$, we may as well take a normalized version of $L$ (\textit{e.g.} $\vert L \vert_{\infty} \leq 1$).
\end{rem}

\begin{rem}[Order of magnitude]
\label{RemOrderMag}
Even if the result holds for any $\beta \geq 1$, the terms $\propto p^{\, \beta-1}$ appearing in some upper bounds of the proof clearly argue in favor of a moderate choice $\beta \in [1, 1+\epsilon]$ for a small $\epsilon > 0$, depending on $p$. In other words, we cannot deviate too much from the Gaussianity in the prior on the direct links. For example in a very high-dimensional setting $(p \sim 10^7)$, choosing $\epsilon = 0.1$ leads to $p^{\, \beta-1} \approx 5$ whereas we may try larger values of $\epsilon$ for the more common high-dimensional settings $p \sim 10^3$ or $p \sim 10^4$. By contrast, we can see that $n_0$ must (at least) grow like $q$ for the theorem to hold, so high-dimensional responses are excluded. However in multiple-output regressions, even when $p$ is extremely large, $q$ generally remains small. According to all these considerations, we may roughly say that, in a high-dimensional setting with respect to $p$,
\begin{equation*}
\Vert \wht - \ts \Vert_{F}\, \lesssim\, \sqrt{\frac{\vert S \vert\, \ln p}{n}}
\end{equation*}
with a large probability, under a suitable regularization of the model. We recognize the usual terms appearing in the error bounds of regressions with high-dimensional covariates, like the $\ell_2$ error of the Lasso (see \textit{e.g.} Chap. 11 of \cite{HastieEtAl15}). This is the same bound as in \cite{YuanZhang14}, but our additional structural penalty restricts $\Lambda$.
\end{rem}

\section{Simulations and real dataset}
\label{SecEmp}

The minimization problem \eqref{DefEst} is solved using a coordinate descent procedure, alternating between the computations of
\begin{equation*}
\whoyy = \arg \min_{\dS_{++}^{\, q}}\, \ell_{\lambda, \mu, \eta}(\oyy, \whoyx) \hsp \text{and} \hsp \whoyx = \arg \min_{\dR^{q \times p}}\, \ell_{\lambda, \mu, \eta}(\whoyy, \oyx).
\end{equation*}
Each step is done by an Orthant-Wise Limited-Memory Quasi-Newton (OWL-QN) algorithm (see \textit{e.g.} \cite{AndrewGao07}). The first subproblem is performed through half-vectorization ($\vech$) to ensure symmetry and we set the objective to $+\infty$ on $\bar{\dS}_{++}^{\, q}$ to ensure positive definiteness of the solution. The coordinate descent is stopped when
\begin{equation*}
\Vert \whoyy^{\, (t)} - \whoyy^{\, (t-1)} \Vert_2\, \leq\, \epsilon \max(1, \Vert \whoyy^{\, (t-1)} \Vert_2) \hsp \text{and} \hsp \Vert \whoyx^{\, (t)} - \whoyx^{\, (t-1)} \Vert_2\, \leq\, \epsilon \max(1, \Vert \whoyx^{\, (t-1)} \Vert_2)
\end{equation*}
following two consecutive iterations $t-1$ and $t$, where $\epsilon > 0$ is a small threshold depending on the desired precision. We are now going to try our method on synthetic data first, and then on a real dataset. We will pay attention to the role played by $\beta$, in particular we will see that it can be useful as well as counterproductive, depending on the situations.

\subsection{Simulations}
\label{SecEmpSim} For each scenario, we first generate i.i.d. standard Gaussian vectors $X_{i} \in \dR^{p}$, then $Y_{i}\in \dR^{q}$ is simulated according to the setting and we estimate $\oyy$ and $\oyx$. From the relations detailed in Section \ref{SecIntro}, we recall that $Y_{i} = B^{\, t}\, X_{i} + E_{i}$ with $E_{i} \sim \cN(0, R)$ is an equivalent formulation, provided that $B = -\oyxt\, \oyyi$ and $R = \oyyi$. In a compact form, we may also write
\begin{equation*}
Y = X B + E \hsp \text{or} \hsp \vec(Y) = (I_{q} \otimes X)\, \vec(B) + \vec(E)
\end{equation*}
where the $i$-th row of $Y$ is $Y_{i}^{\, t}$ and the $i$-th row of $X$ is $X_{i}^{\, t}$. Thus, we can estimate $B$ using the Lasso (Las) and the Group-Lasso (GLas) in the vectorized form, to provide a basis for comparison between our method and the usual penalized methods. The Lasso penalty is obviously $\Vert \vec(B) \Vert_1$ to promote coordinate sparsity while, for the Group-Lasso, we use the penalty $\Vert B_{1} \Vert_2 + \hdots + \Vert B_{p} \Vert_2$ where $B_{i}$ is the $i$-th row of $B$, to promote row sparsity and exclude altogether some predictors from the model. We also implement some variants of our generalized graphical model (GenGm). The case where $\oyy = R^{-1}$ is known and does not need to be estimated is the Oracle (Or) and the case where $\eta=0$ so that $\beta$ has no influence is the classic PGGM (Gm). The case where $\lambda=0$ and $\beta=1$ is called the Spring (Spr) by the authors of \cite{ChiquetEtAl17}. We will focus on structured scenarios. With no structure in $\oyx$, there is no reason why our method should outperform the usual PGGM. In a completely random setting, we have observed that all PGGM procedures perform identically. In fact, a slight gain can be obtained compared to Spr and Gm simply due to the flexibility induced by the additional parameter (Spr and Gm are particular cases of GenGm). However, that clearly cannot counterbalance the extended computational times, and GenGm should not be used for such situations. The calibration of the regularization parameters is made using a cross-validation on a training set of size $n_{t} = 150$ and the accuracy is evaluated thanks to the mean squared prediction error (MSPE) on a validation set of size $n_{v} = 1000$,
\begin{equation}
\label{MSPE}
\textnormal{MSPE} = \frac{\big\Vert Y + X\, \whoyx^{\, t}\, \whoyy^{\, -1}  \big\Vert_{F}^2}{q\, n_{v}}.
\end{equation}
Due to the large amount of treatments, the grids for cross-validation are not very sharp here but they will be carefully refined for the real datasets of the next section. The covariance between the outputs is $R = (r^{\, \vert i-j \vert})_{1\, \leq\, i,\, j\, \leq\, q}$ for $r = \frac{1}{2}$ and we work with $p=100$. Each scenario is repeated $N = 500$ times and GenGm is evaluated with numerous values of $\beta$, from $0.25$ to 2 with a step of $0.25$. The results of the following scenarios are summarized on Figures \ref{FigSimS1}, \ref{FigSimS2} and \ref{FigSimS3} below, respectively.
\begin{itemize}[label=$\rightarrow$]
\item Scenario 1 ($q=1$). We draw $\omega_{i} = \pm \frac{1}{2}$ for $i=1, \hdots, 10$ and we fill 10 randomly selected sections of size $3$ in $\oyx$ with $\omega_{i}$. The remaining part of $\oyx$ is 0.
\item Scenario 2 ($q=2$). We draw $\omega = \pm \frac{1}{2}$ and one randomly selected row of $\oyx$ is filled with $\omega$ while the other is identically 0.
\item Scenario 3 ($q=3$). We draw $\omega_{i} = \pm \frac{1}{2}$ and we fill a randomly selected section of size $30$ on the $i$-th row of $\oyx$ with $\omega_{i}$, for $i=1,2,3$. The remaining part of $\oyx$ is 0.

\end{itemize}
The row structure is promoted by a normalized first finite difference operator
\begin{equation}
\label{MatL}
L = \frac{1}{2} \begin{pmatrix}
1 & - 1 & 0 & \hdots & 0 \\
-1 & 2 & \ddots & \ddots & \vdots \\
0 & \ddots & \ddots & \ddots & 0 \\
\vdots & \ddots & \ddots & 2 & -1 \\
0 & \hdots & 0 & -1 & 1
\end{pmatrix}
\end{equation}
which, through $\oyx\, L\, \oyxt$, tends to penalize the difference between two consecutive values on a same row (as does Fused-Lasso with $\ell_1$ penalty). Yet, the Fused-Lasso is not a suitable alternative to GLas and Las in this precise context because $B = -\oyxt\, \oyyi$ is not supposed to have a row structure even if $\oyx$ has one. For this choice of $L$, one can note that, in the particular case where $R = \diag(\sigma^{\, 2}_1, \hdots, \sigma^{\, 2}_{q})$,
\begin{equation*}
\llangle L, \oyx^{\, t}\, \oyyi\, \oyx \rrangle^{\beta} = \left( \sum_{i=1}^{q} \sigma^{\, 2}_{i} \sum_{j=2}^{p} (\omega_{i,j} - \omega_{i,j-1})^2 \right)^{\! \beta} \geq \sum_{i=1}^{q} \sigma_{i}^{\, 2 \beta} \sum_{j=2}^{p} \vert \omega_{i,j} - \omega_{i,j-1} \vert^{\, 2 \beta}
\end{equation*}
where $\omega_{i,j}$ is the $(i,j)$-th element of $\oyx$, so we may fairly expect that $\beta \geq 1$ is going to strengthen the smoothness of the estimation and to enforce all the more the structuring.

\begin{rem}[Validity of the hypotheses]
\label{RemHyp}
We could as well add a small diagonal element in the matrix $L$ defined above, positive semi-definite but not invertible. The resulting effect would be a negligible ridge-like penalization on the elements of $\oyx$. This is not required for the estimation procedure but useful for Theorem \ref{ThmUppBound} to hold (see \textit{e.g.} \eqref{H1}). Likewise, it seemed interesting to test some settings with $\beta < 1$ even if the theory developped in the paper does not give any guarantee for them, as a basis for comparison.
\end{rem}

\begin{figure}[h!]
\centering
\includegraphics[width=12cm]{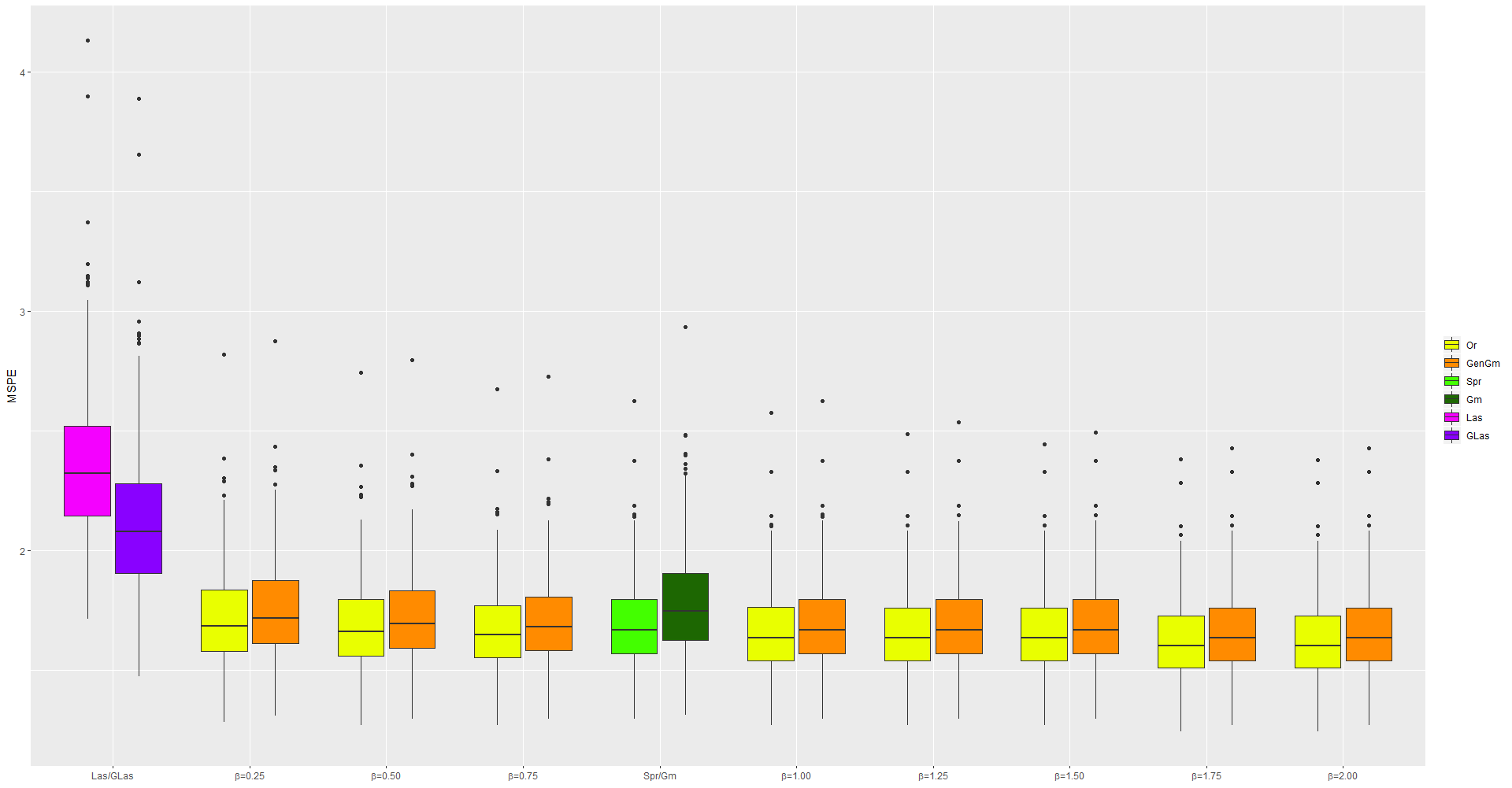}
\caption{Mean squared prediction error for $N=500$ repetitions of the weakly structured Scenario 1.}
\label{FigSimS1}
\end{figure}

\begin{figure}[h!]
\centering
\includegraphics[width=12cm]{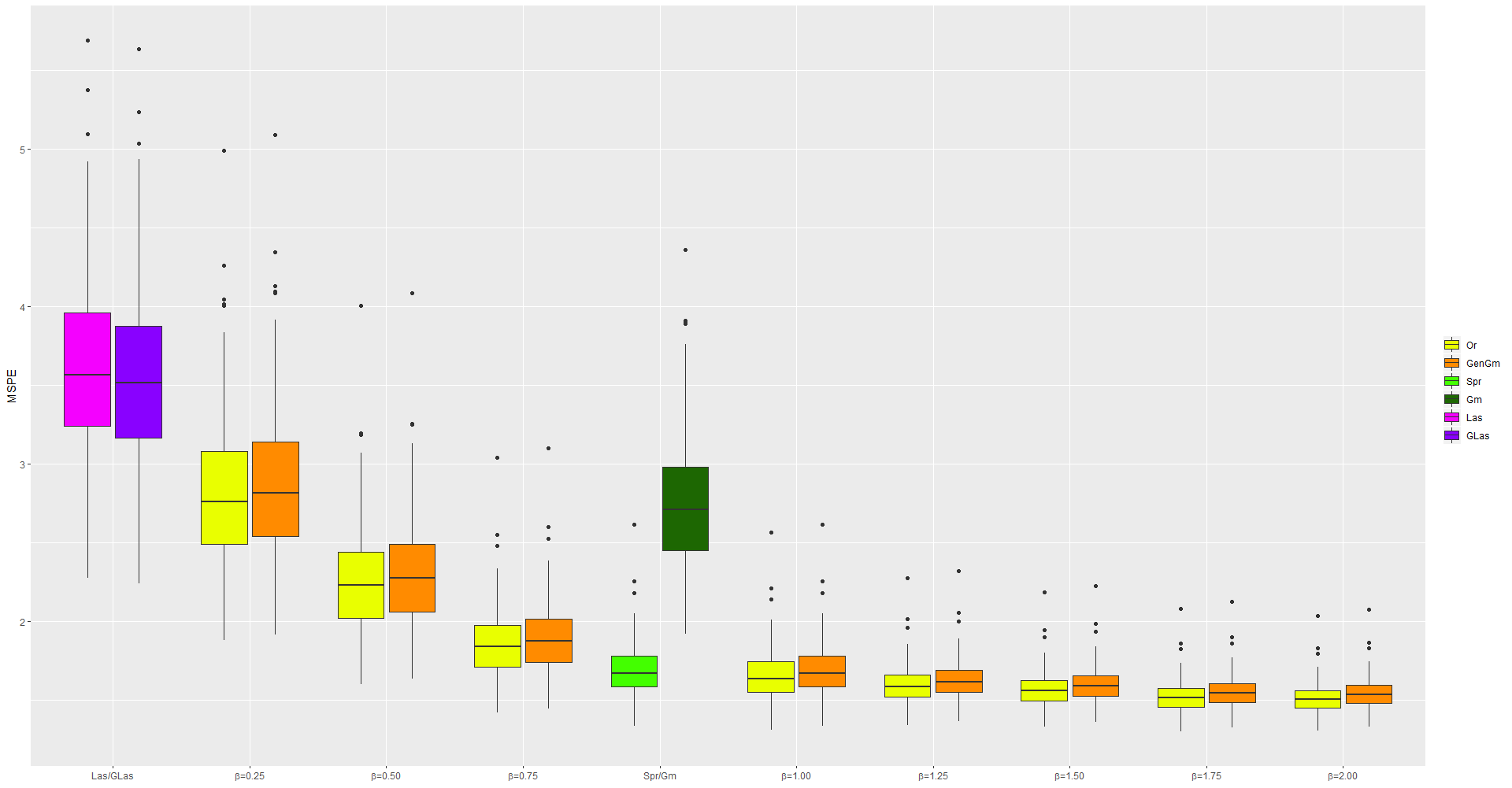}
\caption{Mean squared prediction error for $N=500$ repetitions of the strongly structured Scenario 2.}
\label{FigSimS2}
\end{figure}

\begin{figure}[h!]
\centering
\includegraphics[width=12cm]{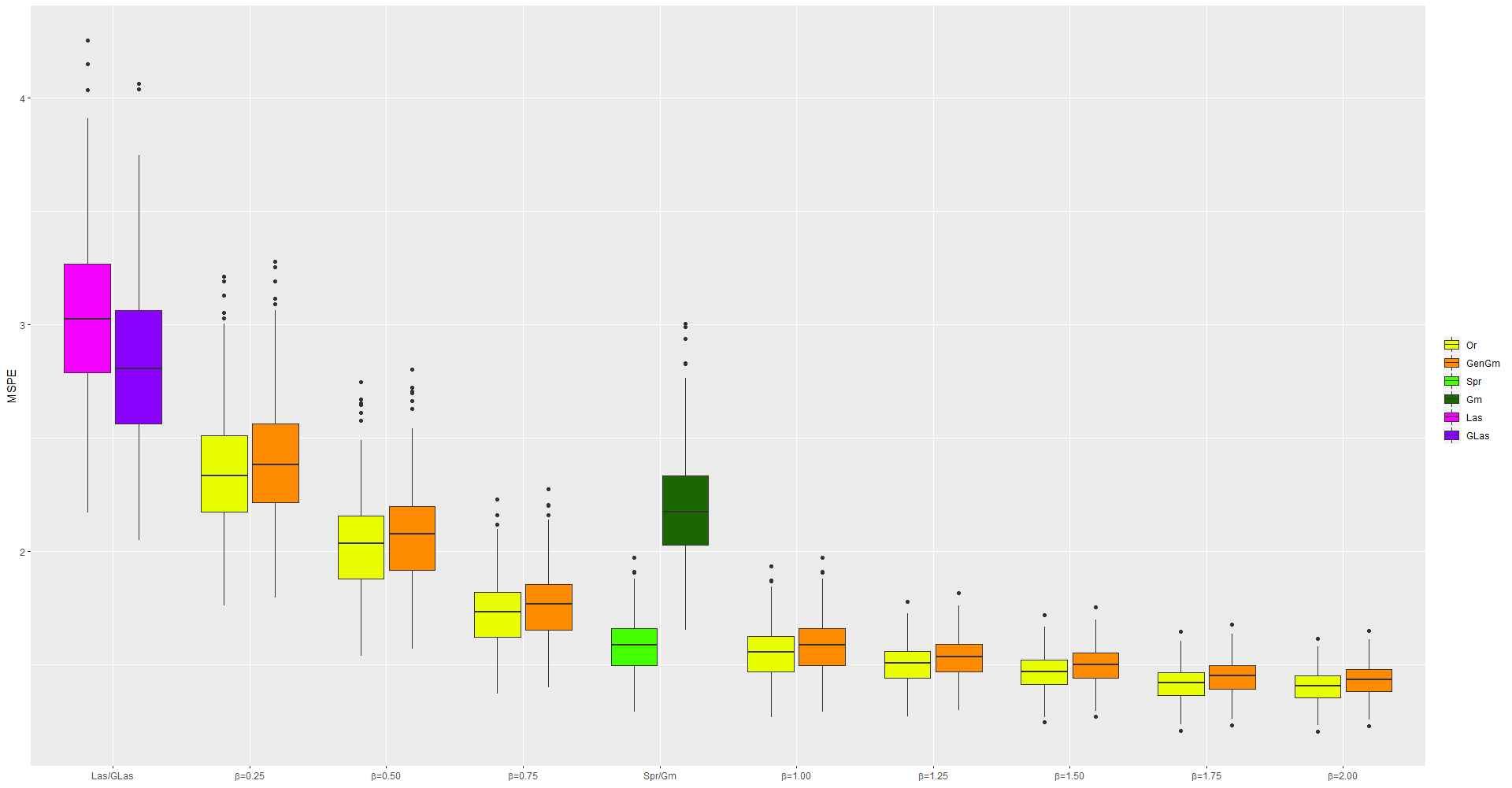}
\caption{Mean squared prediction error for $N=500$ repetitions of the strongly structured Scenario 3.}
\label{FigSimS3}
\end{figure}

First of all, one can observe that Las and GLas are left behind in all our simulations. This is not surprising since the covariance between the outputs cannot be recovered with the standard Lasso, at least for $q \geq 2$. Generally, GLas remains more robust compared to Las, probably due to the high level of sparsity in $\oyx$ approximately passed to $B$ (provided that the covariances in $R$ are small enough), and exploited by the grouping effect. In the weakly structured setting (Scenario 1), we also observe that, as expected, all PGGM procedures perform almost identically, with obviously an advantage for Or (although small, illustrating the accuracy of the estimation). In the strongly structured settings (Scenarios 2 and 3), Gm gives results below the expected level, because it is not designed to promote such layouts. On the contrary, thanks to this choice of $L$ showing here geat efficiency, GenGm and Spr are doing pretty well. Note that, in this context, GenGm with $\beta=1$ is almost the same as Spr since, $q$ being small, $\lambda$ does not play a crucial role. However, some empirical facts draw our attention: the prediction error decreases with $\beta$ to some extent, but the most interesting fact seems to be the simultaneous decrease of its variance. It is likely that the increasing pressure exerted by $\beta$ on the estimation procedure leads to a higher homogeneity in the numerical results, despite the repetitions of random experiments under random settings. In other words, the structuring seems to be strengthened and we also observe an acceleration in the procedure of estimation that logically follows from the latter remarks (especially clear when we compare $\beta=0.25$ and $\beta=2$). On the other hand, for the opposite reason, we notice that the predictions are hardly better than Gm (even worse in some cases), both on average and in terms of variability, for $\beta < 1$, and these simulations tend to undermine such values of the hyperparameter. On the whole, GenGm with $\beta > 1$ might be a sound approach for practitioners who place a high priority on structuring the estimations, even if Remark \ref{RemCompTim} below should probably temper this statement. To conclude, let us consider the strongly structured scenarios with $L = I_{p}$ (without structuring) in the Oracle setting with $\beta=2$, and let us compare the results with those of Figures \ref{FigSimS2} and \ref{FigSimS3}, obtained with the correct version of $L$ given in \eqref{MatL}. The results are displayed on Figure \ref{FigSimUnstruct} where we can see that the benefit of structuring is manifest. Unsurprisingly, the results without structuring are close to those of Gm since $L=I_{p}$ only strengthens the shrinkage effect with ridge-like additional penalties.

\begin{figure}[h!]
\centering
\includegraphics[width=6cm]{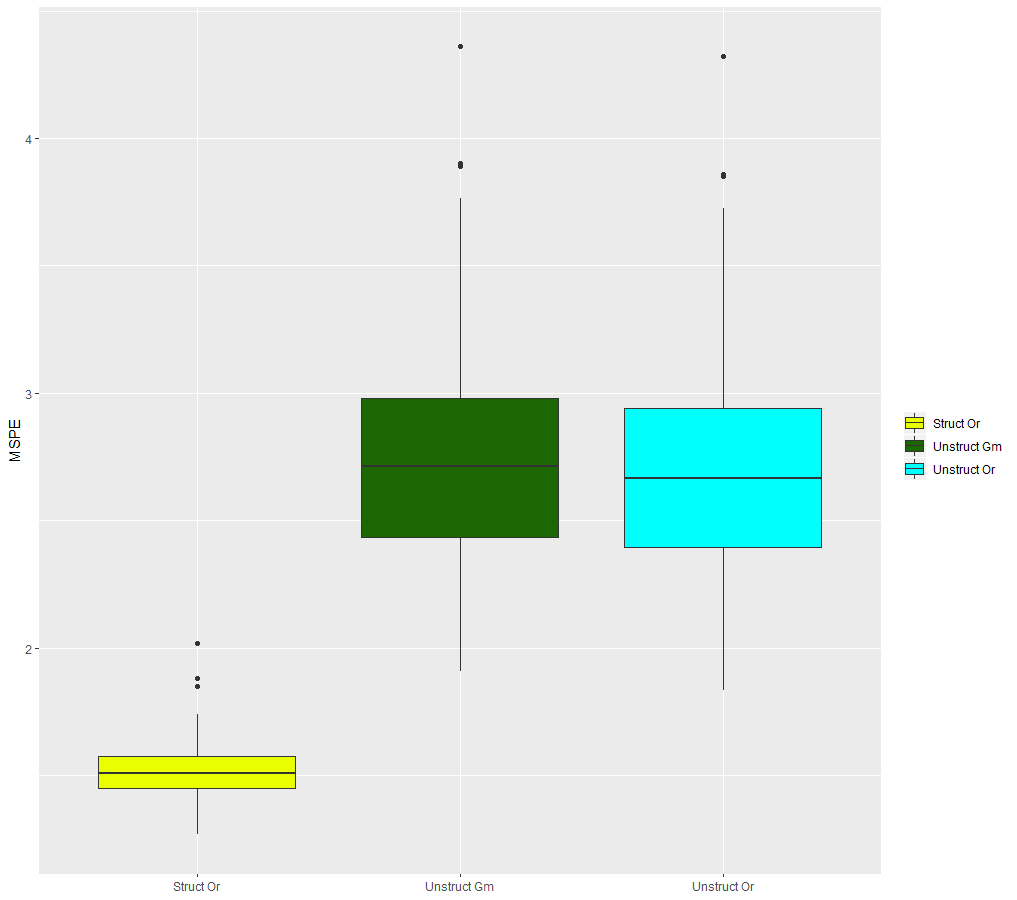} \hsp \includegraphics[width=6cm]{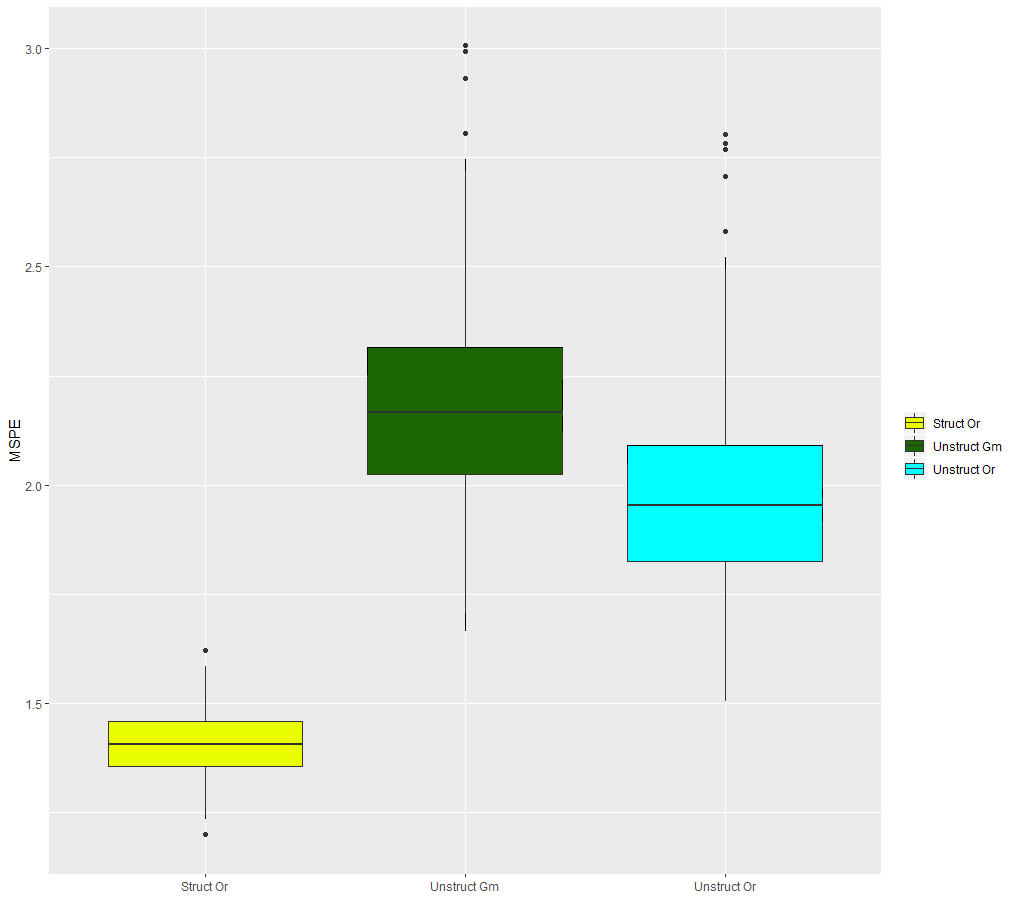}
\caption{Mean squared prediction error for $N=500$ repetitions of the strongly structured Scenario 2 (left) and Scenario 3 (right) for Or, Gm and the unstructured Or ($L = I_{p}$), with $\beta=2$.}
\label{FigSimUnstruct}
\end{figure}

\begin{rem}[Computational time]
\label{RemCompTim}
To estimate $(\oyy, \oyx)$ in the model Spr, the authors of \cite{ChiquetEtAl17} use a very judicious and efficient method relying, in each step of the coordinate descent procedure, on a direct computation of the estimation of $\oyy$ together with an Elastic-Net estimation of $\oyx$. This is possible for $\lambda=0$ and $\beta=1$, but unfortunately cannot be implemented in the general setting. As a result, computational times remain an issue that should be paid attention to.
\end{rem}

\begin{rem}[Oracle-type errors]
\label{RemOracle}
The mean value of the estimation errors $\Vert \whoyx - \oyx \Vert_{F}^2$ leads to the same kind of observations for the models being compared in the simulations. But the minimal prediction error does not always coincide with an optimal support recovery due to the shrinkage effect on the estimation of $\oyx$. The so-called $F$-score is given by
\begin{equation*}
F = \frac{2\, p_{r}\, r_{e}}{p_{r}+r_{e}} \hsp \text{where} \hsp p_{r} = \frac{\textnormal{TP}}{\textnormal{TP} + \textnormal{FP}} \hsp \text{and} \hsp r_{e} = \frac{\textnormal{TP}}{\textnormal{TP} + \textnormal{FN}}
\end{equation*}
are the \textit{precision} and the \textit{recall}, respectively, and where T/F and P/N stand for true/false and positive/negative. In the strongly structured scenarios, $F$ is generally located between $0.60$ and $0.65$, and a deeper analysis shows that a proportion of more than $0.99$ of true non-zero values are recovered (that is, the part of the true active set $S$ related to $\oyx$). If the models are not calibrated to reach the best prediction error but the best $F$-score, $F$ regularly exceeds $0.90$, at least for the structured procedures.
\end{rem}

Nevertheless, Scenarios 2 and 3 are very strongly structured, more than one would expect from an unknown underlying generating process, and the real dataset of the next section is going to highlight the fact that the improvement may be hardly noticeable with respect to $\beta$, even in terms of prediction error variance. But we will see that $\beta$ can still be useful for variable selection.

\subsection{A real dataset}
\label{SecEmpReal}

The dataset available as \texttt{CanadianWeather} in the \texttt{R} package \texttt{fda} contains daily temperature and precipitation at 35 different locations in Canada, averaged over annual reports starting in 1960 and ending in 1994 (see \textit{e.g.} \cite{RamsaySilverman06}). We intend to look at the direct links between the minimal and maximal rainfall (on the log$_{10}$ scale) and the temperature pattern in the 35 weather stations, so as to identify the times of the year that have a strong effect on rainfall (positive as well as negative). In this context, $n=35$, $q=2$ and $p=365$. Figure \ref{FigPrecDistr} shows temperature and log-precipitation measured over a year in Montreal, chosen as an example, together with the empirical distribution of the minimal and maximal log-precipitation for the 35 weather stations. We can note that, since the data are averaged over numerous years, outliers are unlikely even for the extremes (min and max).

\begin{figure}[h!]
\centering
\includegraphics[width=9cm, height=6cm]{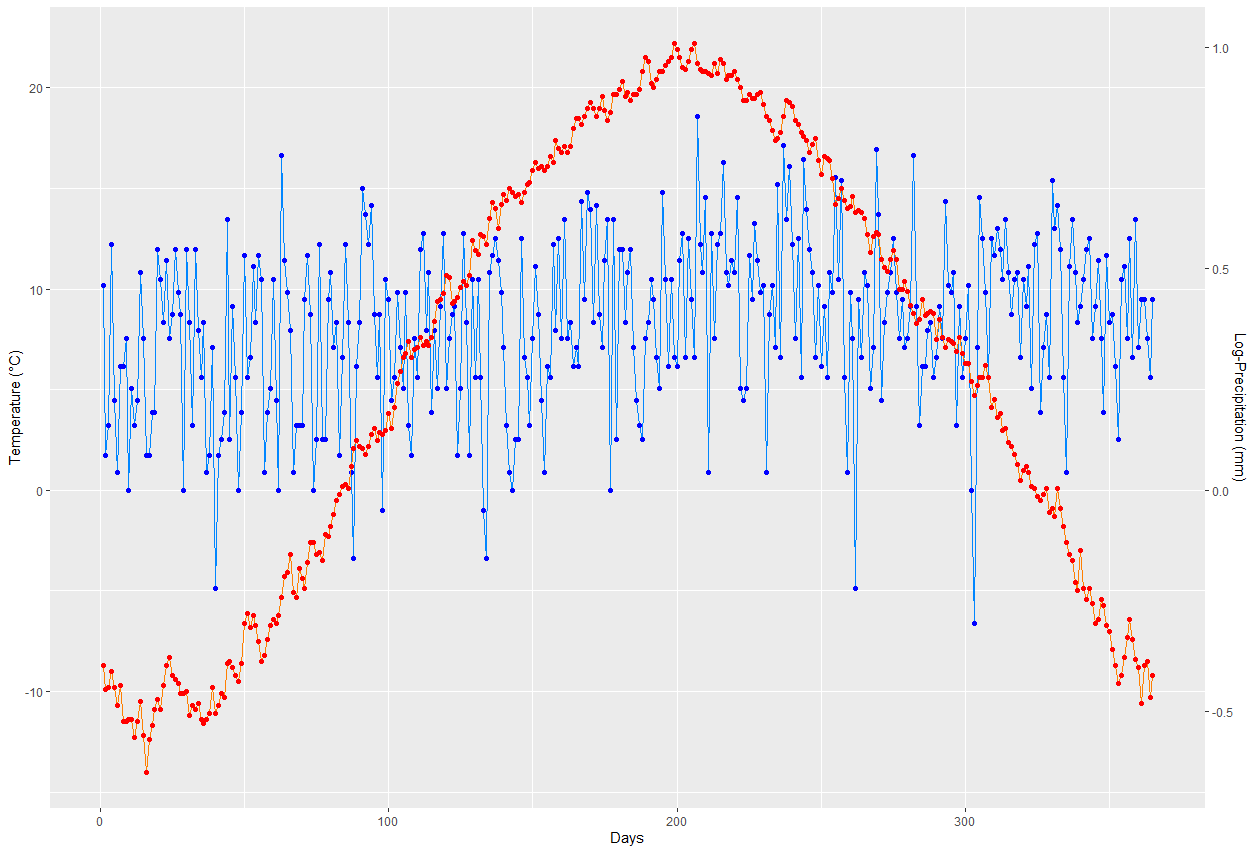} \hsp \includegraphics[width=6cm]{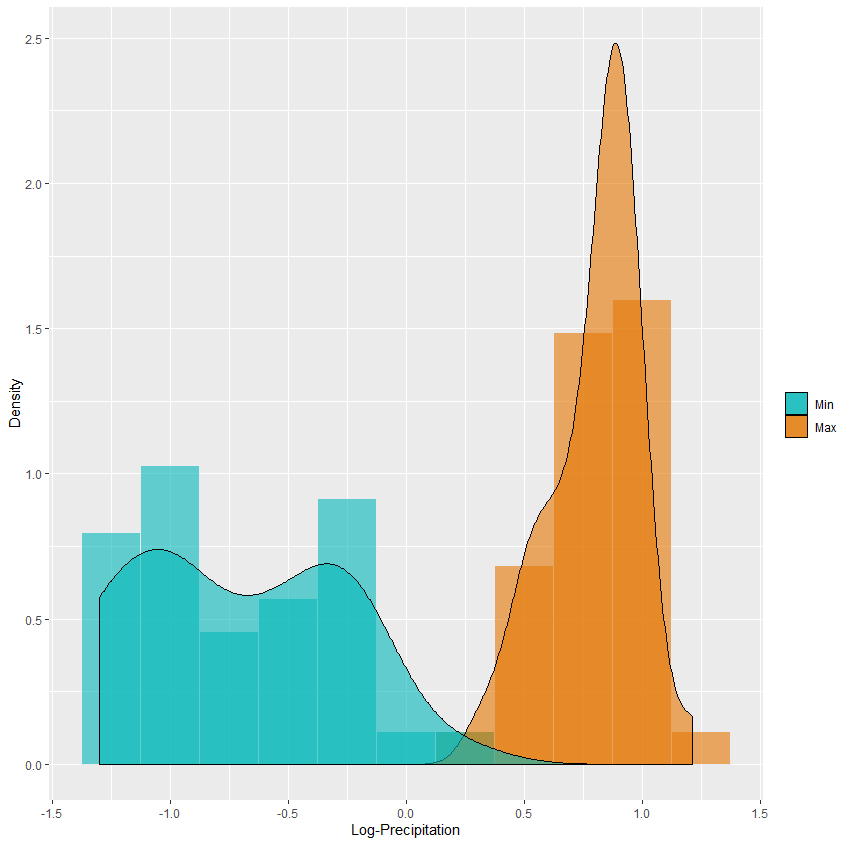}
\caption{Temperature and log-precipitation measured over a year in Montreal (left). Empirical distribution of the minimal and maximal log-precipitation for the 35 weather stations (right).}
\label{FigPrecDistr}
\end{figure}

Some authors (see \textit{e.g.} \cite{Slawski12}) have already highlighted the pertinence of using the matrix $L$ defined in \eqref{MatL} in this dataset, because the predictors are ordered temporally so that the selection of isolated days instead of relevant sequences of days seems an unreliable procedure for statistical interpretation. To assess the models, we repeat $N=100$ times the following experiment: $n_{t} = 25$ observations are randomly selected for calibration (\textit{via} 2-fold cross-validation) and estimation, the remaining $n_{v}=10$ observations are used to compute the MSPE \eqref{MSPE} related to the prediction of the minimum ($\min_{\text{p}}$) and maximum ($\max_{\text{p}}$) precipitation. We can see on Figure \ref{FigPrecMSPE} that all structured PGGM perform almost identically, with the phenomenon described in the previous section still visible but to a lesser extent. We can even notice that structuring is hardly beneficial for this dataset, from a purely numerical point of view. This conclusion can also be found in \cite{Slawski12}, where the author compares the structured Elastic-Net with unstructured alternatives to predict the $0.25$-, $0.50$- and $0.75$-quantiles of the log-precipitation. But we will see that, in terms of variable selection and statistical interpretation, $L$ and $\beta$ still have a substancial role to play.

\begin{figure}[h!]
\centering
\includegraphics[width=6cm]{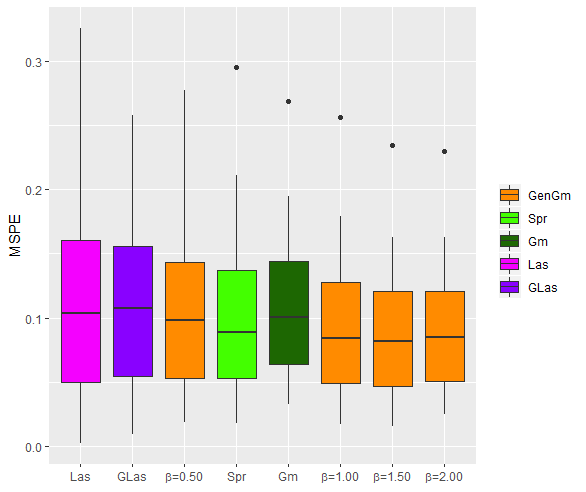}
\caption{Mean squared prediction error for $N=100$ repetitions of the experiment. GenGm for $\beta \in \{ 0.5, 1, 1.5, 2\}$ is compared with Spr, Gm, Las and GLas.}
\label{FigPrecMSPE}
\end{figure}

The point is that the best prediction error does not usually coincide with a sparse solution (see Remark \ref{RemOracle} above). On the basis of the MSPE, most of the time we must retain $\mu \ll 10^{-2}$ and only a few direct links are set to zero. To look for sequences of days directly related to $\min_{\text{p}}$ and $\max_{\text{p}}$, let us constraint $\mu \geq 10^{-2}$ and focus on variable selection. The active set of $\oyx$ is evaluated on the basis of $n_{t} = 25$ randomly chosen observations. The experiment is repeated $N=100$ times, and the locations having a frequency of occurrence that exceeds $0.5$ are retained (or, equivalently, those whose estimates have a non-zero median). This can be seen as a measure of variable importance. The results are given on Figures \ref{FigPrecMin} and \ref{FigPrecMax} below for $\min_{\text{p}}$ and $\max_{\text{p}}$, respectively, with a fixed set of regularization parameters and increasing values of $\beta$. The objective is to show the influence of the latter, all other things being equal. Note that, since we retain $\lambda=0$ in these experiments, GenGm for $\beta=1$ coincides with Spr.

\begin{figure}[h!]
\centering
\includegraphics[width=14cm]{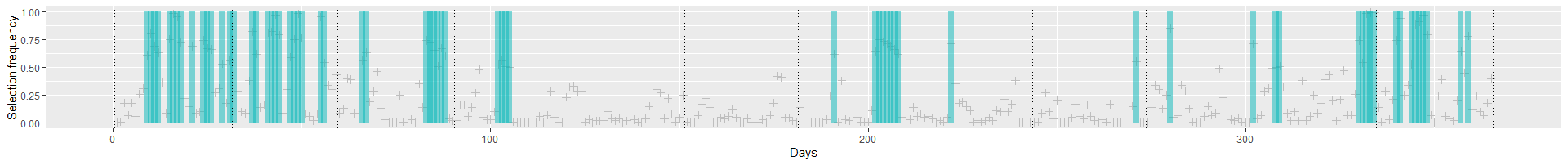}
\includegraphics[width=14cm]{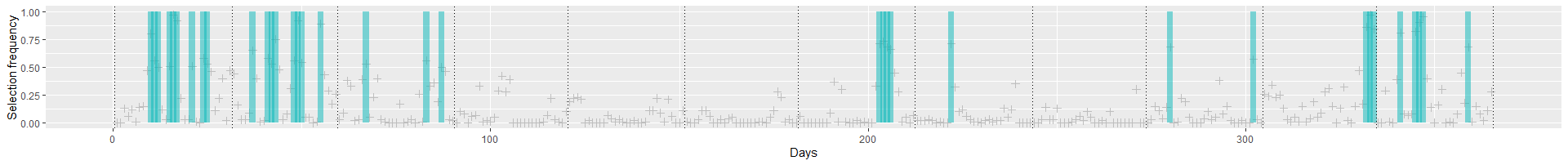}
\includegraphics[width=14cm]{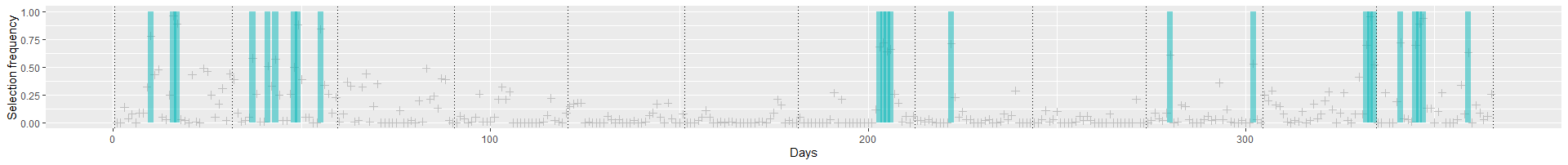}
\includegraphics[width=14cm]{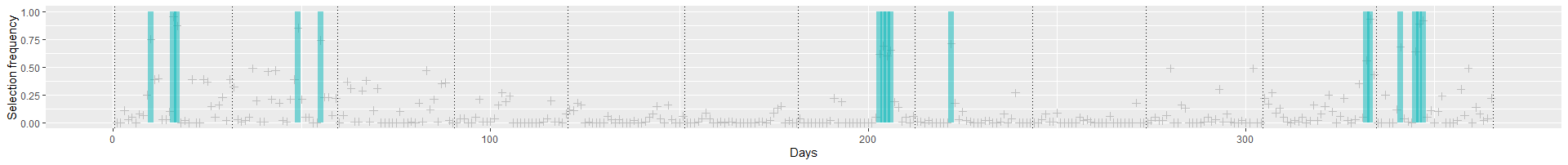}
\caption{Variable selection for $\min_{\text{p}}$ by GenGm with $(\lambda, \mu, \eta)=(0, 0.05, 1)$ and, from top to bottom, $\beta \in \{0.5, 1, 1.5, 2\}$. The colored areas highlight the days having a frequency of occurrence, represented by gray crosses, that exceeds $0.5$ in the $N=100$ repetitions of the experiment. Dotted lines divide the panel into months.}
\label{FigPrecMin}
\end{figure}

\begin{figure}[h!]
\centering
\includegraphics[width=14cm]{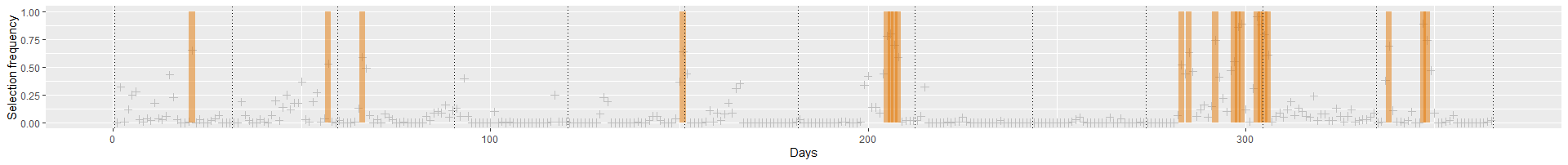}
\includegraphics[width=14cm]{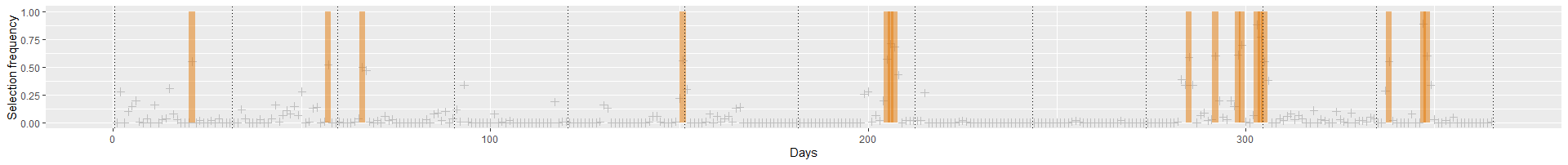}
\includegraphics[width=14cm]{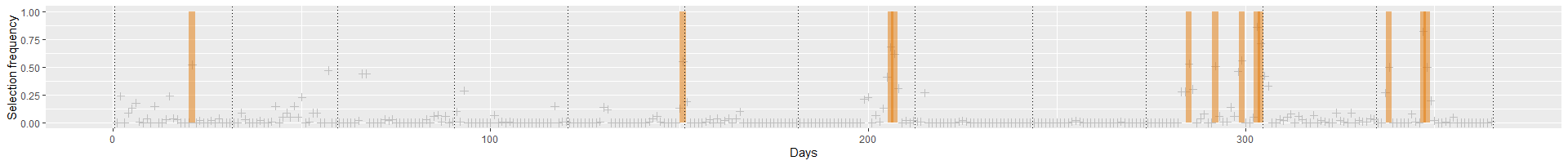}
\includegraphics[width=14cm]{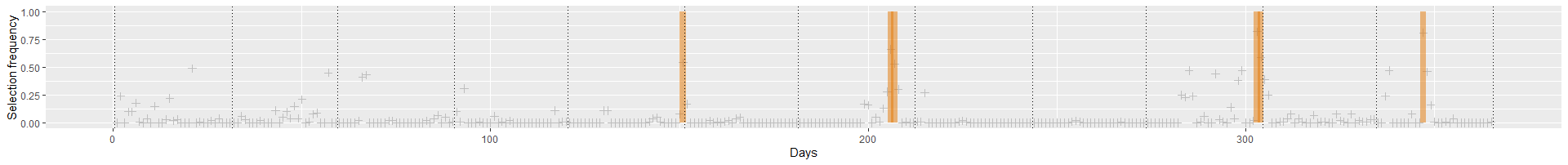}
\caption{Variable selection for $\max_{\text{p}}$ by GenGm with $(\lambda, \mu, \eta)=(0, 0.05, 1)$ and, from top to bottom, $\beta \in \{0.5, 1, 1.5, 2\}$. The colored areas highlight the days having a frequency of occurrence, represented by gray crosses, that exceeds $0.5$ in the $N=100$ repetitions of the experiment. Dotted lines divide the panel into months.}
\label{FigPrecMax}
\end{figure}

We observe that the increasing pressure exerted by $\beta$ on the estimation procedure tends to refine the selection by giving priority to the most important variables and by dropping the others much more easily, at the cost of prediction results: 
we are undoubtedly in a selection process. The sequence of inclusions
\begin{equation*}
\wh{S}_{\beta_2} \subset \wh{S}_{\beta_1} \hsp \text{for $\beta_1 < \beta_2$}
\end{equation*}
that we observe for the estimated active sets is clearly a guarantee of quality for the selected variables. The median values of the estimated direct links between the temperature of the days and the pair $(\min_{\text{p}}, \max_{\text{p}})$ are represented on Figure \ref{FigPrecEst} together with the estimated regression coefficients, for $\beta=2$. We detect sequences of influent days in November, December, January and February, especially related to $\min_{\text{p}}$, positively at the end of the year and negatively at the beggining. This is broadly consistent with the analysis of \cite{Slawski12} -- even if the responses are not extremes but quantiles in this reference -- with however two differences: the regression coefficients associated with $\max_{\text{p}}$ are much lower, and an activity is also detected between July and August. The main explanation, at least for the first of them, probably lies in the use of graphical models that take into account the correlation between responses. Indeed, as can be seen on Figure \ref{FigPrecEstVar} which gives an overview of the estimation of $R$ obtained from the repeated experiments, a non-zero correlation is detected between the responses ($\approx 0.32$). The influence of November and December on all quantiles and that of January and February on the 0.75-quantile in \cite{Slawski12} might actually be an artificial effect of the correlation with the 0.25-quantile. This is what our study suggests by highlighting $\min_{\text{p}}$ compared to $\max_{\text{p}}$. From this point of view, the interest of graphical models is particularly obvious.

\begin{figure}[h!]
\centering
\includegraphics[width=14cm]{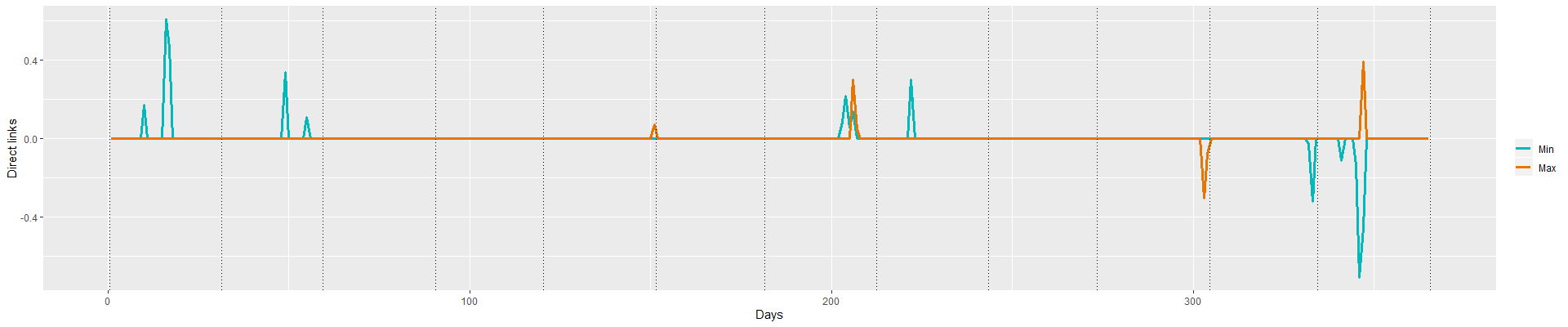}
\includegraphics[width=14.1cm]{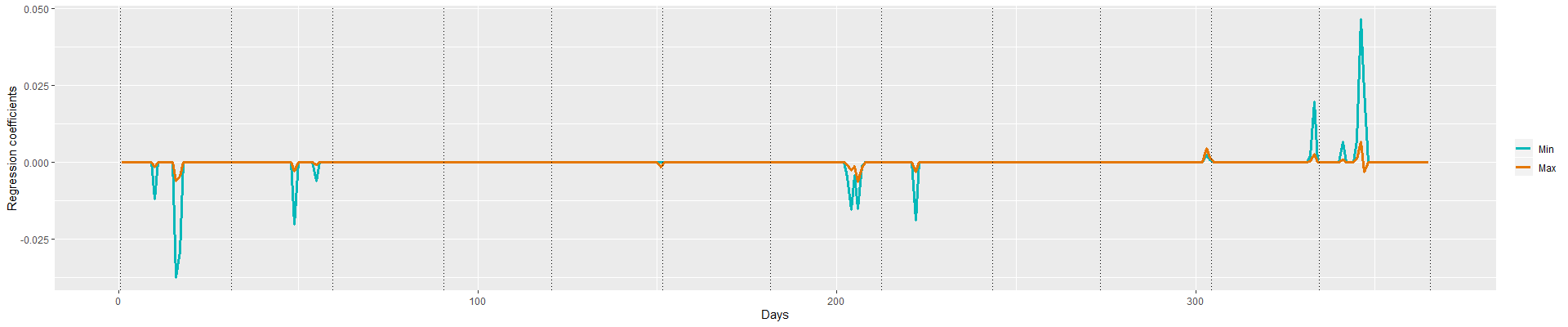}
\caption{Estimated direct links (top) and regression coefficients (bottom) for the pair $(\min_{\text{p}}, \max_{\text{p}})$ by GenGm with $(\lambda, \mu, \eta)=(0, 0.05, 1)$ and $\beta = 2$, after the $N=100$ experiments. Dotted lines divide the panel into months.}
\label{FigPrecEst}
\end{figure}

\begin{figure}[h!]
\centering
\includegraphics[width=4cm]{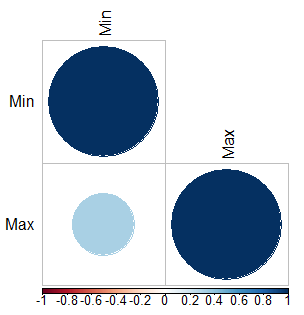}
\caption{Estimated correlation between $\min_{\text{p}}$ and $\max_{\text{p}}$ by GenGm with $(\lambda, \mu, \eta)=(0, 0.05, 1)$ and $\beta = 2$, after the $N=100$ experiments. The off-diagonal entry is approximately $0.32$.}
\label{FigPrecEstVar}
\end{figure}

Let us also mention that, interestingly enough, we notice that the role of $\eta$ tends to depreciate for the large values of $\beta$. For example, for the same regularization parameters $(\lambda,\mu)=(0, 0.05)$ and $\beta=2$, the difference between the estimated active sets for $\eta=0.1$ and $\eta=1$ is almost negligible (depending on the experiments, between 1 and 3 days are concerned, on average). Based on these studies and observations, we might conclude that $\beta$ is insignificant when we are interested in the best prediction error on a validation set (even counterproductive with respect to computational times, \textit{e.g.} compared to Spr), whereas it seems to have a substancial role to play when focusing on selection, by accelerating the discrimination of variables. In the first case, $\eta$ has to be carefully adjusted while in the second case, $\beta$ will quickly help to reach the desired sparsity.

\begin{rem}[Structure matrix]
\label{RemStructMat}
For the simulations and the real dataset, we have used the popular first finite difference operator given in \eqref{MatL}. Other examples can be found in the literature, like the promotion of a genetic distance for genomic selection in \textit{Brassica napus} \cite{ChiquetEtAl17} or the bidimensional discretization of the Laplacian to work on handwritten digit recognition \cite{Slawski12}. More generally, $L$ can be used in a classic Bayesian prior supposed to promote some covariance structure on the direct links, with no `physical' structuring in mind (like temporal, spatial or genetic proximity).
\end{rem}

\section{Conclusion}
\label{SecConclu}

In conclusion, our work is a generalization of \cite{YuanZhang14}, using the same technical tools to establish an upper bound on the estimation error when a prior on the direct links generates an additional structural penalty in the objective, provided that the model is suitably regularized. Our work is also an improvement of \cite{ChiquetEtAl17} since, while being inspired by the methodology of the authors, we generalize the prior and give some theoretical guarantees. The empirical study shows that the hyperparametrization in the prior, although more expensive in adjusting the parameters, is likely to refine the selection results but clearly, this does not appear as a crucial improvement compared to the two previous points. Let us conclude the paper by highlighting two weaknesses that might be trails for future studies. On the one hand, the Laplace distribution is often used as a prior in the Bayesian Lasso (see \textit{e.g.} Sec. 6.1 of \cite{HastieEtAl15}). However, our reasonings do not allow $\beta=1/2$, which may correspond to a multivariate Laplace distribution on the direct links. Combined with the first finite difference operator $L$, the choice $\beta=1/2$ could generate a Fused-Lasso-type penalty. In this regard, it would be challenging and interesting to obtain some theoretical guarantees for $\beta \geq 1/2$ and not only for $\beta \geq 1$, even if our probably too brief simulation study does not encourage the choice of $\beta < 1$. On the other hand, $\lambda=0$ is a natural choice when $q$ is small (this is in particular the configuration of \cite{ChiquetEtAl17}), not to mention that it is computionally faster. But, the proof of our theorem needs $\lambda > c_{\lambda}\, h_{a} > 0$ to hold. We think that a reasoning enabling to deal with $\lambda = 0$ should also be beneficial to the study. More generally, it would be instructive to consider a very high-dimensional setting ($p \gg n$ and not only $p \sim 10^2$ although always larger than $n,$ as in our experiments). Such studies should follow with omic data.

\bigskip

\noindent \textbf{Acknowledgements and Fundings.} The authors warmly thank the two anonymous reviewers for the careful reading and for making numerous useful corrections to improve the paper. We thank ALM (Angers Loire M\'etropole) and the ICO (Institut de Canc\'erologie de l'Ouest) for the financial support. This work is partially financed through the ALM grant and the ``Programme op\'erationnel r\'egional FEDER-FSE Pays de la Loire 2014-2020" noPL0015129 (EPICURE). The authors also thank Mario Campone (project leader and director of the ICO), Mathilde Colombi\'e (scientific coordinator of EPICURE clinical trial) and Fadwa Ben Azzouz, biomathematician in Bioinfomics, for the initiation, the coordination and the smooth running of the project.

\section{Technical proofs}
\label{SecPro}

We start in a first part by some useful linear algebra lemmas that will be repeatedly used subsequently, well-known for most of them. In a second part, we prove the joint convexity of the objective and our main result.

\subsection{Linear algebra}
\label{SecProLinAlg}

\begin{lem}
\label{LemProdSDP}
Let $A \in \dS_{+}^{\, d}$ and $U \in \dR^{d \times \ell}$. Then, $U^{\, t} A U \in \dS_{+}^{\, \ell}$.
\end{lem}
\begin{proof}
Since $A$ is symmetric with non-negative eigenvalues, there is an orthogonal matrix $P$ such that $A = P D P^{\, t}$ with $D = \diag(\textnormal{sp}(A))\in \dS_{+}^{\, d}$. Thus, for all $v \in \dR^{\ell}$, it follows that $\langle v, U^{\, t} A\, U\, v \rangle = \Vert D^{1/2}\, P^{\, t}\, U\, v \Vert^2 \geq 0$. 
\end{proof}

\begin{lem}
\label{LemProdVP}
Let $A \in \dS_{++}^{\, d}$ and $B \in \dS_{+}^{\, d}$. Then for all $i$, $\lambda_{i}(AB) \geq 0$.
\end{lem}
\begin{proof}
The equality $A B = A^{1/2}\, (A^{1/2}\, B\, A^{1/2})\, A^{-1/2}$ shows that $A B$ and $A^{1/2}\, B\, A^{1/2}$ are similar, so they must share the same eigenvalues. From Lemma \ref{LemProdSDP}, $\lambda_{i}(A^{1/2}\, B\, A^{1/2}) \geq 0$ .
\end{proof}

\begin{lem}
\label{LemProdEncVP}
Let $A \in \dS_{+}^{\, d}$ and $B \in \dS_{+}^{\, d}$. Then,
\begin{equation*}
\lmin(A)\, \tr(B)\, \leq\, \tr(A B)\, \leq\, \lmax(A)\, \tr(B).
\end{equation*}
\end{lem}
\begin{proof}
Since $A - \lmin(A) I_{d} \in \dS_{+}^{\, d}$ and $B \in \dS_{+}^{\, d}$,
\begin{equation*}
\tr((A - \lmin(A) I_{d})\, B) = \tr(B^{1/2}\, (A - \lmin(A) I_{d})\, B^{1/2}) \geq 0
\end{equation*}
from Lemma \ref{LemProdSDP}, thus $\tr(A B) \geq \lmin(A)\, \tr(B)$. The other inequality is obtained through $\lmax(A) I_{d} - A \in \dS_{+}^{\, d}$.
\end{proof}

\begin{lem}
\label{LemProdMinMaxVP}
Let $A \in \dS_{++}^{\, d}$ and $B \in \dS_{+}^{\, d}$. Then,
\begin{equation*}
\lmin(A)\, \lmin(B)\, \leq\, \lmin(A B) \hsp \text{and} \hsp \lmax(A B)\, \leq\, \lmax(A)\, \lmax(B).
\end{equation*}
\end{lem}
\begin{proof}
On the one hand, $\lmax(A B) \leq \Vert A B \Vert_2 \leq \Vert A \Vert_2\, \Vert B \Vert_2 = \lmax(A)\, \lmax(B)$, since $A$ and $B$ are symmetric and since, from Lemma \ref{LemProdVP} and by hypothesis, all eigenvalues appearing in the relation are non-negative. Suppose now that $B$ is invertible so that both $A^{-1}$ and $B^{-1}$ belong to $\dS_{++}^{\, d}$. Then, $\lmax( (A B)^{-1} ) \leq \lmax(A^{-1})\, \lmax(B^{-1})$ and this immediately gives $\lmin(A B) \geq \lmin(A)\, \lmin(B)$. If $B$ is not invertible, the relation trivially holds since we still have $\lmin(A B) \geq 0$ from Lemma \ref{LemProdVP}.
\end{proof}

\begin{lem}
\label{LemProdTrace}
Let $A \in \dS_{+}^{\, d}$ and $U \in \dR^{d \times \ell}$. Then,
\begin{equation*}
\lmin(A)\, \Vert U \Vert_{F}^2\, \leq\, \tr(U^{\, t} A U)\, \leq\, \lmax(A)\, \Vert U \Vert_{F}^2. 
\end{equation*}
\end{lem}
\begin{proof}
Denote by $u_{i}$ the $i$-th column of $U$. It is not hard to see that the $i$-th diagonal element of $U^{\, t} A U$ is $u_{i}^{\, t} A\, u_{i} \geq \lmin(A)\, \Vert u_{i} \Vert^2 \geq 0$. Thus,
\begin{equation*}
\tr(U^{\, t} A U) = \sum_{i=1}^{\ell} u_{i}^{\, t} A\, u_{i} \geq \lmin(A) \sum_{i=1}^{\ell} \Vert u_{i} \Vert^2 = \lmin(A)\, \Vert U \Vert_{F}^2.
\end{equation*}
The upper bound stems from $0 \leq u_{i}^{\, t} A\, u_{i} \leq \lmax(A)\, \Vert u_{i} \Vert^2$.
\end{proof}

\begin{lem}
\label{LemWeyl}
Let $A$ and $B$ be symmetric matrices of same dimensions. Then,
\begin{equation*}
\lmin(A) + \lmin(B)\, \leq\, \lmin(A+B) \hsp \text{and} \hsp \lmax(A+B)\, \leq\, \lmax(A) + \lmax(B).
\end{equation*}
\end{lem}
\begin{proof}
These are just two special cases of Weyl inequalities. We refer the reader to Thm. 4.3.1 of \cite{HornJohnson92}, for example.
\end{proof}

\subsection{Convexity of the objective}
\label{SecProConv}

We know from Prop. 1 of \cite{YuanZhang14} and the convexity of the elementwise $\ell_1$ norm that $L_{n}(\oyy, \oyx) - \eta\, \llangle L, \oyxt\, \oyyi\, \oyx \rrangle^{\beta}$ is itself convex, but it remains to show that this is still the case with the additional smooth penalty.

\subsubsection*{Proof of Proposition \ref{PropConv}}
Recall that $\Theta = \dS_{++}^{\, q} \times \dR^{q \times p}$ and consider the mapping $\Phi : \Theta \rightarrow \dS_{+}^{\, p}$ defined as
\begin{equation*}
\forall\, (A,B) \in \Theta, \hsp \Phi(A, B) = B^{\, t} A^{-1} B.
\end{equation*}
We can already note from Lemma \ref{LemProdSDP} that $\tr(\Phi(A, B)) \geq 0$. Moreover, for all $0 \leq h \leq 1$ and all $Z_{i} = (A_{i},B_{i}) \in \Theta$, $i=1,2$, it is easy to see that
\begin{equation}
\label{ConvSchur}
S_{h}(Z_1,Z_2) = h\, \Phi(Z_1) + (1-h)\, \Phi(Z_2) - \Phi(h Z_1 + (1-h) Z_2)
\end{equation}
is the Schur complement of $h A_1 + (1-h) A_2$ in the matrix
\begin{equation}
\label{ConvM}
M_{h}(Z_1, Z_2) = h \begin{pmatrix}
A_1 & B_1 \\
B_1^{\, t} & B_1^{\, t}\, A_1^{-1} B_1
\end{pmatrix} + (1-h) \begin{pmatrix}
A_2 & B_2 \\
B_2^{\, t} & B_2^{\, t}\, A_2^{-1} B_2
\end{pmatrix}.
\end{equation}
But the decomposition
\begin{equation*}
\begin{pmatrix}
A^{1/2} & A^{-1/2}\, B \\
0 & 0
\end{pmatrix}^{\! t} \begin{pmatrix}
A^{1/2} & A^{-1/2}\, B \\
0 & 0
\end{pmatrix} = \begin{pmatrix}
A & B \\
B^{\, t} & B^{\, t}\, A^{-1} B
\end{pmatrix}
\end{equation*}
directly shows that $M_{h}(Z_1, Z_2)$ in \eqref{ConvM} is symmetric and positive semi-definite. It is well-known (see \textit{e.g.} Appendix A.5.5 of \cite{BoydVandenberghe04}) that in that case, the Schur complement \eqref{ConvSchur} must also be positive semi-definite. Consequently, for $Z_{i} = (\Omega_{i,yy}, \Omega_{i,yx} L^{1/2})$, $i=1,2$, taking the trace of $S_{h}(Z_1,Z_2)$ and considering $\beta \geq 1$,
\begin{eqnarray*}
\llangle L, P_{h}^{\, t}\, Q_{h}^{-1} P_{h} \rrangle^{\beta} & = & (\tr(\Phi(h Z_1 + (1-h) Z_2)))^{\beta} \\
 & \leq & (h\, \tr(\Phi(Z_1)) + (1-h)\, \tr(\Phi(Z_2)))^{\beta} \\
 & = & (h\, \llangle L, \Omega_{1,yx}^{\, t}\, \Omega_{1,yy}^{-1}\, \Omega_{1,yx} \rrangle + (1-h)\, \llangle L, \Omega_{2,yx}^{\, t}\, \Omega_{2,yy}^{-1}\, \Omega_{2,yx} \rrangle)^{\beta} \\
 & \leq & h\, \llangle L, \Omega_{1,yx}^{\, t}\, \Omega_{1,yy}^{-1}\, \Omega_{1,yx} \rrangle^{\beta} + (1-h)\, \llangle L, \Omega_{2,yx}^{\, t}\, \Omega_{2,yy}^{-1}\, \Omega_{2,yx} \rrangle^{\beta}
\end{eqnarray*}
where $P_{h} = h\, \Omega_{1,yx} + (1-h)\, \Omega_{2,yx}$ and $Q_{h} = h\, \Omega_{1,yy} + (1-h)\, \Omega_{2,yy}$. This convexity inequality concludes the proof. \hfill \qedsymbol

\subsection{Theoretical guarantees}
\label{SecProThm}

\subsubsection*{Proof of Theorem \ref{ThmUppBound}}
Let $R_{n}(\theta)$ be the the smooth part of the objective \eqref{LikHPGGM},
\begin{eqnarray}
\label{SmoothRn}
R_{n}(\theta) & = & -\ln \det(\oyy) + \llangle \vary, \oyy \rrangle + 2\, \llangle \covyx, \oyx \rrangle \nonumber \\
& & \hsp \hsp \hsp \hsp +~ \llangle \varx, \oyxt\, \oyyi\, \oyx \rrangle + \eta\, \llangle L, \oyx^{\, t}\, \oyyi\, \oyx \rrangle^{\beta}.
\end{eqnarray}
For any $\theta \in \Theta$ and $t \in \dR$, by a Taylor expansion,
\begin{equation}
\label{TaylorRn}
R_{n}(\ts + t\, (\theta-\ts)) = R_{n}(\ts) + t\, \llangle \nabla R_{n}(\ts), \theta-\ts \rrangle + e_{t}(\theta, \ts)
\end{equation}
for some second-order error term $e_{t}(\theta, \ts)$. Consider the reparametrization
\begin{equation}
\label{DefPhi}
\phi(t) = R_{n}(\ts + t\, (\theta-\ts))
\end{equation}
so that $\phi^{\prime}(0) = \llangle \nabla R_{n}(\ts), \theta-\ts \rrangle$. Let $\dyy = \oyy - \oyys$ and $\dyx = \oyx - \oyxs$, let also $\dt = \theta - \ts$ in a compact form. The estimation error is denoted
\begin{equation}
\label{EstErr}
\dvt = \wht - \ts = (\whoyy - \oyys, \whoyx - \oyxs) = (\dvyy, \dvyx).
\end{equation}
Before we start the actual proof, some additional lemmas are needed. They constitute a local study in a sort of $r^{*}$-neighborhood of $\ts$ that we define as
\begin{equation}
\label{DefNR}
N_{r, \alpha}(\ts) = \big\{ \theta \in \Theta,\, \Vert \dt \Vert_{F} \leq r^{*} \text{ and } \vert [\dt]_{\bar{S}} \vert_1 \leq \alpha \vert [\dt]_{S} \vert_1 \big\}.
\end{equation}
Our strategy can be summarized as follows:
\begin{itemize}[label=$\rightarrow$]
\item (Lemma \ref{LemInegSmSp}) Show that there exists a configuration for the regularization parameters $(\lambda, \mu, \eta)$ so that the estimation error satisfies $\vert [\dvt]_{\bar{S}} \vert_1 \leq \alpha \vert [\dvt]_{S} \vert_1$ for some $\alpha > 0$.
\item (Lemma \ref{LemMajErr}) Find some $r^{*} > 0$ and $\gamma_{r,\eta,\beta,p} > 0$ such that $e_1(\theta, \ts) > \gamma_{r,\eta,\beta,p} \Vert \dt \Vert_{F}^2$ as soon as $\theta \in N_{r, \alpha}(\ts)$.
\item (Lemma \ref{LemBorneErr}) Exploit this result to show that the estimation error must also satisfy $\Vert \dvt \Vert_{F} \leq r^{*}$ provided that $\max\{h_{a}, h_{b} \}$ is small enough.
\item (Lemma \ref{LemBorneFinale}) Conclude that the theorem holds with high probability, provided that $n$ is large enough.
\end{itemize}
For the sake of readability, we refer the reader to the Appendix for the numerous constants that are about to appear in the following lemmas and proofs. Thereafter, $N_{r, \alpha}(\ts)$ will always refer to $\alpha$ in \eqref{CstAlpha} and $r^{*}$ in \eqref{CstR}, while the second hypothesis \eqref{H2} given below is to be assumed with the smallest integer greater than $s_{\alpha}$ in \eqref{CstSa}. This is a random hypothesis, which will be controlled with a probability, related to the proximity between the empirical covariance and the true covariance of the predictors, since we recall that $S^{\, (n)}$ has no reason to be an excellent approximation of $\Sigma^{*}$ when $p \gg n$. This is also assumed by the authors of \cite{YuanZhang14}, it is a kind of restricted isometry propertie (RIP), well-known in high-dimensional studies. In particular, we will see through Lemma \ref{LemBorneFinale} that it is satisfied with high probability provided that $n$ is large enough.
\begin{equation}
\label{H2}\tag{H$_2$}
\forall\, u \neq 0~ \text{such that}~ \vert u \vert_0 \leq \lceil s_{\alpha} \rceil, \hsp \frac{1}{2}\, u^{\, t}\, \sxxs\, u\, \leq\, u^{\, t}\, \varx\, u\, \leq\, \frac{3}{2}\, u^{\, t}\, \sxxs\, u.
\end{equation}
\begin{equation*}
\text{In addition,} \hsp \lmax(\oyxs\, \varx\, \oyxts)\, \leq\, \frac{7}{5}\, \lmax(\oyxs\, \sxxs\, \oyxts).
\end{equation*}
The next two lemmas give some bounds for expressions that will appear repeatedly.

\begin{lem}
\label{LemBornesVPS}
Under \eqref{H1} and \eqref{H2}, for all $\theta \in N_{r, \alpha}(\ts)$, we have the bound
\begin{equation*}
\lmax(\oyyi\, \oyx\, \varx\, \oyxt)\, \leq\, \omss
\end{equation*}
where $\omss$ is given in \eqref{BornesVP}. In addition,
\begin{equation*}
\tr(\dyx\, \varx\, \dyxt)\, \geq\, \frac{\lmin(\sxxs)}{10}\, \Vert \dyx \Vert_{F}^2.
\end{equation*}
\end{lem}
\begin{proof}
Similar reasonings may be found in the proofs of Lem. 1-2 of \cite{YuanZhang14}. We simply reworked the constants to make them stick to our study.
\end{proof}

\begin{lem}
\label{LemBornesVPL}
Under \eqref{H1}, for all $\theta \in N_{r, \alpha}(\ts)$, we have the bounds
\begin{equation*}
\lmin(\oyyi\, \oyx\, L\, \oyxt)\, \geq\, \omli \hsp \text{and} \hsp \lmax(\oyyi\, \oyx\, L\, \oyxt)\, \leq\, \omls
\end{equation*}
where $\omli$ and $\omls$ are given in \eqref{BornesVP}. As a corollary,
\begin{equation*}
p\, \omli\, \leq\, \llangle L, \oyxt\, \oyyi\, \oyx \rrangle\, \leq\, p\, \omls.
\end{equation*}
\end{lem}
\begin{proof}
From Lemmas \ref{LemProdSDP} and \ref{LemWeyl},
\begin{eqnarray*}
2\, \lmin(\oyx\, L\, \oyxt) & \geq & 2\, \big( \lmin(\oyxs\, L\, \oyxts) + \lmin(\dyx\, L\, \oyxts + \oyxs\, L\, \dyxt) \big) \\
 & \geq & 2\, \big( \lmin(\oyxs\, L\, \oyxts) - \Vert \dyx\, L\, \oyxts + \oyxs\, L\, \dyxt \Vert_{2} \big) \\
 & \geq & 2\, \big( \lmin(\oyxs\, L\, \oyxts) - 2\, \Vert \dyx \Vert_{F}\, \Vert L\, \oyxts \Vert_{2} \big) ~ \geq ~ \lmin(\oyxs\, L\, \oyxts)
\end{eqnarray*}
as soon as $\Vert \dyx \Vert_{F} \leq r^{*}$. From Lemma \ref{LemProdMinMaxVP}, we get
\begin{equation*}
\lmin(\oyyi\, \oyx\, L\, \oyxt) \geq \frac{\lmin(\oyx\, L\, \oyxt)}{\lmax(\oyy)} \geq \frac{\lmin(\oyxs\, L\, \oyxts)}{4\, \lmax(\oyys)}
\end{equation*}
where the inequality in the denominator comes from $\lmax(\oyy) \leq \lmax(\oyys) + \lmax(\dyy)$, \textit{via} Lemma \ref{LemWeyl}, and the fact that $\lmax(\dyy) \leq \Vert \dyy \Vert_{F} \leq r^{*} \leq \lmax(\oyys)$. For the upper bound, a similar logic gives, with Lemma \ref{LemProdTrace},
\begin{eqnarray*}
\sqrt{\lmax(\oyx\, L\, \oyxt)} & \leq & \sqrt{\lmax(\oyxs\, L\, \oyxts)} + \sqrt{\tr(\dyx\, L\, \dyxt)} \\
 & \leq & \sqrt{\lmax(\oyxs\, L\, \oyxts)} + \Vert \dyx \Vert_{F}\, \sqrt{\lmax(L)} ~ \leq ~ \sqrt{2\, \lmax(\oyxs\, L\, \oyxts)}
\end{eqnarray*}
for $\Vert \dyx \Vert_{F} \leq r^{*}$. It follows from Lemma \ref{LemProdMinMaxVP} that
\begin{equation*}
\lmax(\oyyi\, \oyx\, L\, \oyxt) \leq \frac{\lmax(\oyx\, L\, \oyxt)}{\lmin(\oyy)} \leq \frac{4\, \lmax(\oyxs\, L\, \oyxts)}{\lmin(\oyys)}
\end{equation*}
where the inequality in the denominator comes from $\lmin(\oyy) \geq \lmin(\oyys) + \lmin(\dyy)$, \textit{via} Lemma \ref{LemWeyl}, and the fact that $2\, \lmin(\dyy) \geq -2\, \Vert \dyy \Vert_{F} \geq -2\, r^{*} \geq -\lmin(\oyys)$. The corollary that concludes the lemma is now immediate.
\end{proof}

\begin{lem}
\label{LemInegSmSp}
Assume that $\lambda$, $\mu$ and $\eta$ are chosen according to the configuration of the theorem. Then, under \eqref{H1}, the estimation error satisfies
\begin{equation*}
\vert [ \dvt]_{\bar{S}} \vert_1 \leq \alpha\, \vert [ \dvt ]_{S} \vert_1
\end{equation*}
where $\alpha > 0$ is given in \eqref{CstAlpha}.
\end{lem}
\begin{proof}
Taking $t=1$ in the Taylor expansion \eqref{TaylorRn} with $\theta = \wht$ and considering the definition of $\phi$ in \eqref{DefPhi}, by convexity,
\begin{equation*}
R_{n}(\wht) - R_{n}(\ts) \geq \phi^{\prime}(0).
\end{equation*}
The first derivative of $\phi$ will be explicitely computed in \eqref{Der1Phi}. For $t=0$, we find
\begin{eqnarray*}
\phi^{\prime}(0) & = & -\llangle \oyyis, \dvyy \rrangle  + \llangle \vary, \dvyy \rrangle + 2\, \llangle \covyx, \dvyx \rrangle \nonumber \\
 & & \hsp \hsp \hsp +~ 2\, \llangle \varx, \oyxts\, \oyyis\, \dvyx \rrangle - \llangle \varx, \oyxts\, \oyyis\, \dvyy\, \oyyis\, \oyxs \rrangle \nonumber \\
 & & \hsp \hsp \hsp +~ \eta \beta\, s_{L}^{\, \beta-1} \big[ 2\, \llangle L, \oyxts\, \oyyis\, \dvyx \rrangle - \llangle L, \oyxts\, \oyyis\, \dvyy\, \oyyis\, \oyxs \rrangle \big] \\
 & = & \llangle A_{n} + \eta \beta\, s_{L}^{\, \beta-1} C_{A}, \dvyy \rrangle + \llangle B_{n} + \eta \beta\, s_{L}^{\, \beta-1} C_{B}, \dvyx \rrangle
\end{eqnarray*}
where $s_{L}$ is given in \eqref{CstL}, where, through the blockwise relations \eqref{InvBloc}, we recognize the random matrices $A_{n}$ (with max norm $h_{a}$) and $B_{n}$ (with max norm $h_{b}$) defined in \eqref{MatA} and \eqref{MatB}, and where, coming from the structural regularization term,
\begin{equation*}
C_{A} = -  \oyyis\, \oyxs\, L\, \oyxts\, \oyyis \hsp \text{and} \hsp C_{B} = 2\, \oyyis\, \oyxs\, L.
\end{equation*}
Whence it follows from the well-known relation $\vert \tr(M_1 M_2) \vert \leq \vert M_1 \vert_{\infty}\, \vert M_2 \vert_1$, where $M_1$ and $M_2$ are compatible matrices, that
\begin{equation*}
\phi^{\prime}(0) \geq -\frac{\lambda}{c_{\lambda}}\, \vert \dvyy \vert_1 - \eta \beta\, s_{L}^{\, \beta-1} \ell_{a}\, \vert \dvyy \vert_1 - \frac{\mu}{c_{\mu}}\, \vert \dvyx \vert_1 - \eta \beta\, s_{L}^{\, \beta-1} \ell_{b}\, \vert \dvyx \vert_1
\end{equation*}
making use of the constants \eqref{CstL}, $\lambda \geq c_{\lambda}\, h_{a}$ and $\mu \geq c_{\mu}\, h_{b}$. For the sake of clarity, let
\begin{equation*}
\Delta_{n}(\theta, \ts) = R_{n}(\theta) + \lambda\, \vert \oyy \vert_1^{-} + \mu\, \vert \oyx \vert_1 - R_{n}(\ts) - \lambda\, \vert \oyys \vert_1^{-} - \mu\, \vert \oyxs \vert_1.
\end{equation*}
For all $\theta \in \Theta$,
\begin{eqnarray*}
\vert \oyy \vert_1^{-} - \vert \oyys \vert_1^{-} & = & \vert [\oyys + \dyy]_{S} \vert_1^{-} + \vert [\dyy]_{\bar{S}} \vert_1^{-} - \vert [\oyys]_{S} \vert_1^{-} \\
 & \geq & \big\vert \vert [\oyys]_{S} \vert_1^{-} - \vert [\dyy]_{S} \vert_1^{-} \big\vert + \vert [\dyy]_{\bar{S}} \vert_1^{-} - \vert [\oyys]_{S} \vert_1^{-} \\
 & \geq & \vert [\dyy]_{\bar{S}} \vert_1 - \vert [\dyy]_{S} \vert_1
\end{eqnarray*}
from the triangle inequality and the fact that, as $\oyys$ is positive definite, the diagonal must belong to $S$, \textit{i.e.} $(j,j) \in S$ for all $1 \leq j \leq q$ so that any square matrix $M$ of size $q$ is such that $[M]_{\bar{S}}$ has diagonal elements all equal to zero. A similar bound obviously holds for $\vert \oyx \vert_1 - \vert \oyxs \vert_1$. Now, a straightforward calculation shows that
\begin{equation}
\label{Delta}
\Delta_{n}(\wht, \ts) \geq \underline{c}\, \big( \vert [\dvyy]_{\bar{S}} \vert_1 + \vert [\dvyx]_{\bar{S}} \vert_1 \big) - \overline{c}\, \big( \vert [\dvyy]_{S} \vert_1 + \vert [\dvyx]_{S} \vert_1 \big)
\end{equation}
where
\begin{equation*}
\overline{c} = \max\left\{ \frac{(c_{\lambda}+1) \lambda}{c_{\lambda}} + \eta \beta\, s_{L}^{\, \beta-1} \ell_{a}, \frac{(c_{\mu}+1) \mu}{c_{\mu}} + \eta \beta\, s_{L}^{\, \beta-1} \ell_{b} \right\}
\end{equation*}
and
\begin{equation*}
\underline{c} = \min\left\{ \frac{(c_{\lambda}-1) \lambda}{c_{\lambda}} - \eta \beta\, s_{L}^{\, \beta-1} \ell_{a}, \frac{(c_{\mu}-1) \mu}{c_{\mu}} - \eta \beta\, s_{L}^{\, \beta-1} \ell_{b} \right\}.
\end{equation*}
Thus, provided that $\underline{c} > 0$, which is stated in the configuration of the theorem, it only remains to note that, necessarily,
\begin{equation*}
\Delta_{n}(\wht, \ts) \leq 0
\end{equation*}
since $\wht$ is the global minimizer of $\theta \mapsto R_{n}(\theta) + \lambda\, \vert \oyy \vert_1^{-} + \mu\, \vert \oyx \vert_1$. The identification of $\alpha$ given in \eqref{CstAlpha} easily follows.
\end{proof}

\begin{lem}
\label{LemMajErr}
Under \eqref{H1} and \eqref{H2}, the second-order error term of \eqref{TaylorRn} satisfies, for $t=1$ and all $\theta \in N_{r, \alpha}(\ts)$,
\begin{equation*}
e_1(\theta, \ts) > \gamma_{r,\eta,\beta,p}\, \Vert \dt \Vert_{F}^2
\end{equation*}
where $\gamma_{r,\eta,\beta,p} > 0$ is given in \eqref{CstGam}.
\end{lem}
\begin{proof}
From the definition of $\phi$ in \eqref{DefPhi} and the fact that $\phi^{\prime}(0) = \llangle \nabla R_{n}(\ts), \theta-\ts \rrangle$, there exists $h \in~ ]0,1[$ satisfying
\begin{equation}
\label{ErrTaylor}
e_1(\theta, \ts) = \frac{1}{2}\, \phi^{\prime\prime}(h).
\end{equation}
To simplify the calculations, let
\begin{equation}
\label{DefUxy}
u_{L} = \llangle L, \oyxt\, \oyyi\, \oyx \rrangle.
\end{equation}
We are going to study the behavior of $R_{n}(\oyy, \oyx)$ in the directions $\oyy = \oyys + t\, \dyy$ and $\oyx = \oyxs + t\, \dyx$ through $\phi(t)$, where we recall that $\dyy = \oyy - \oyys$ and $\dyx = \oyx - \oyxs$. One can see that $\phi(t)$ moves from $R_{n}(\oyy, \oyx)$ to $R_{n}(\oyys, \oyxs)$ as $t$ decreases from 1 to 0. The first derivative is
\begin{eqnarray}
\label{Der1Phi}
\phi^{\prime}(t) & = & -\llangle \oyyi, \dyy \rrangle + \llangle \vary, \dyy \rrangle + 2\, \llangle \covyx, \dyx \rrangle \nonumber \\
 & & \hsp \hsp \hsp +~ 2\, \llangle \varx, \oyxt\, \oyyi\, \dyx \rrangle - \llangle \varx, \oyxt\, \oyyi\, \dyy\, \oyyi\, \oyx \rrangle \nonumber \\
 & & \hsp \hsp \hsp +~ \eta \beta\, u_{L}^{\, \beta-1} \big[ 2\, \llangle L, \oyxt\, \oyyi\, \dyx \rrangle - \llangle L, \oyxt\, \oyyi\, \dyy\, \oyyi\, \oyx \rrangle \big].
\end{eqnarray}
The second derivative is tedious to write but straightforward to establish,
\begin{eqnarray}
\label{Der2Phi}
\phi^{\prime\prime}(t) & = & 
\llangle \oyyi, \dyy\, \oyyi\, \dyy \rrangle + 2\, \big[ \llangle \varx, \dyxt\, \oyyi\, \dyx \rrangle - 2\, \llangle \varx, \oyxt\, \oyyi\, \dyy\, \oyyi\, \dyx \rrangle \nonumber \\
 & & \hsp \hsp \hsp +~ \llangle \varx, \oyxt\, \oyyi\, \dyy\, \oyyi\, \dyy\, \oyyi\, \oyx \rrangle \big] \nonumber \\
 & & \hsp \hsp \hsp +~ 2\, \eta \beta\, u_{L}^{\, \beta-1} \big[ \llangle L, \dyxt\, \oyyi\, \dyx \rrangle - 2\, \llangle L, \oyxt\, \oyyi\, \dyy\, \oyyi\, \dyx \rrangle \nonumber \\
 & & \hsp \hsp \hsp +~ \llangle L, \oyxt\, \oyyi\, \dyy\, \oyyi\, \dyy\, \oyyi\, \oyx \rrangle \big] \nonumber \\
 & & \hsp \hsp \hsp +~ \eta \beta (\beta-1)\, u_{L}^{\, \beta-2} \big[ 2\, \llangle L, \oyxt\, \oyyi\, \dyx \rrangle - \llangle L, \oyxt\, \oyyi\, \dyy\, \oyyi\, \oyx \rrangle \big]^2.
\end{eqnarray}
First, from the combination of Lemmas \ref{LemProdSDP} and \ref{LemBornesVPL}, we clearly have $u_{L} \geq 0$. We also note that $0 \leq \Vert \frac{2}{c} M_1 - c M_2 \Vert^2_{F} = \frac{4}{c^2}\, \Vert M_1 \Vert^2_{F} - 4\, \llangle M_1, M_2 \rrangle + c^2\, \Vert M_2 \Vert^2_{F}$ for any $c \neq 0$ and any matrices $M_1$ and $M_2$ of same dimensions. It follows, after some reorganizations, that for any $c \neq 0$ and $d \neq 0$,
\begin{eqnarray*}
\phi^{\prime\prime}(t) & \geq & \llangle \oyyi, \dyy\, \oyyi\, \dyy \rrangle \\
 & & \hsp \hsp \hsp +~ c_1\, \llangle \oyyi, \dyx\, \varx\, \dyxt \rrangle + c_2\, \llangle \varx, \oyxt\, \oyyi\, \dyy\, \oyyi\, \dyy\, \oyyi\, \oyx \rrangle \\
 & & \hsp \hsp \hsp +~ \eta \beta\, u_{L}^{\, \beta-1} \big[ d_1\, \llangle \oyyi, \dyx\, L\, \dyxt \rrangle + d_2\, \llangle L, \oyxt\, \oyyi\, \dyy\, \oyyi\, \dyy\, \oyyi\, \oyx \rrangle \big]
\end{eqnarray*}
where $c_1 = 2 - \frac{4}{c^2}$, $c_2 = 2 - c^2$, $d_1 = 2 - \frac{4}{d^2}$ and $d_2 = 2 - d^2$. Here we exploited the previous inequality twice, $u_{L} \geq 0$ and $\beta \geq 1$. From Lemmas \ref{LemProdSDP}, \ref{LemProdEncVP}, \ref{LemBornesVPS} and \ref{LemBornesVPL}, using $\textnormal{sp}(M_1\, M_2) = \textnormal{sp}(M_2\, M_1)$ for square matrices $M_1$ and $M_2$, we obtain
\begin{equation*}
\llangle L, \oyxt\, \oyyi\, \dyy\, \oyyi\, \dyy\, \oyyi\, \oyx \rrangle \leq \omls\, \llangle \oyyi, \dyy\, \oyyi\, \dyy \rrangle
\end{equation*}
where $\omls$ is defined in \eqref{BornesVP}. Replacing $L$ by $\varx$ and $\omls$ by $\omss$, a similar bound obviously holds. Suppose that $c$ and $d$ are chosen so that $c_1 > 0$, $d_1 > 0$, $c_2 < 0$ and $d_2 < 0$. Then,
\begin{eqnarray*}
\phi^{\prime\prime}(t) & \geq & \llangle \oyyi, \dyy\, \oyyi\, \dyy \rrangle\, \big[ 1 - \vert c_2 \vert\, \omss - \eta \beta\, u_{L}^{\, \beta-1} \vert d_2 \vert\, \omls \big] \\
 & & \hsp \hsp \hsp +~ c_1\, \llangle \oyyi, \dyx\, \varx\, \dyxt \rrangle + \eta \beta\, u_{L}^{\, \beta-1} d_1\, \llangle \oyyi, \dyx\, L\, \dyxt \rrangle \\
 & \geq & \llangle \oyyi, \dyy\, \oyyi\, \dyy \rrangle\, \big[ 1 - \vert c_2 \vert\, \omss - \eta \beta\, (p\, \omls)^{\beta-1} \vert d_2 \vert\, \omls \big] \\
 & & \hsp \hsp \hsp +~ c_1\, \llangle \oyyi, \dyx\, \varx\, \dyxt \rrangle + \eta \beta\, (p\, \omli)^{\beta-1} d_1\, \llangle \oyyi, \dyx\, L\, \dyxt \rrangle.
\end{eqnarray*}
Now choose $\epsilon_{S} > 0$ and $\epsilon_{L} > 0$ small enough so that $\epsilon_{S}\, \omss + \eta \beta\, p^{\, \beta-1}\, \omls^{\, \beta}\, \epsilon_{L} < 1$ and fix $c = \sqrt{2 + \epsilon_{S}}$ and $d = \sqrt{2 + \epsilon_{L}}$. We finally obtain
\begin{equation}
\label{InegDerSec}
\phi^{\prime\prime}(t) \geq a_1\, \llangle \oyyi, \dyy\, \oyyi\, \dyy \rrangle + a_2\, \llangle \oyyi, \dyx\, \varx\, \dyxt \rrangle + a_3\, \llangle \oyyi, \dyx\, L\, \dyxt \rrangle
\end{equation}
where these positive constants are respectively given by
\begin{equation*}
a_1 = 1 - \epsilon_{S}\, \omss - \eta \beta\, p^{\, \beta-1}\, \omls^{\, \beta}\, \epsilon_{L}, \hsp a_2 = \frac{2\, \epsilon_{S}}{2+\epsilon_{S}} \hsp \text{and} \hsp a_3 = \eta \beta\, (p\, \omli)^{\beta-1} \frac{2\, \epsilon_{L}}{2+\epsilon_{L}}.
\end{equation*}
The combination of Lemmas \ref{LemProdSDP}, \ref{LemProdEncVP} and \ref{LemProdTrace} gives, uniformly in $t \in [0, 1]$,
\begin{equation*}
\llangle \oyyi, \dyy\, \oyyi\, \dyy \rrangle \geq \lmin(\oyyi)\, \tr(\dyy\, \oyyi\, \dyy) \geq \frac{\Vert \dyy \Vert_{F}^2}{4\, \lmax^2(\oyys)}
\end{equation*}
where the inequality in the denominator comes from $\lmax(\oyy) \leq 2\, \lmax(\oyys)$ already established in the proof of Lemma \ref{LemBornesVPL}. Similarly,
\begin{equation*}
\llangle \oyyi, \dyx\, L\, \dyxt \rrangle \geq \lmin(\oyyi)\, \tr(\dyx\, L\, \dyxt) \geq \frac{\lmin(L)\, \Vert \dyx \Vert_{F}^2}{2\, \lmax(\oyys)}.
\end{equation*}
Lemma \ref{LemBornesVPS} directly enables to bound the last term,
\begin{equation*}
\llangle \oyyi, \dyy\, \varx\, \dyy \rrangle \geq \lmin(\oyyi)\, \tr(\dyx\, \varx\, \dyxt) \geq \frac{\lmin(\sxxs)\, \Vert \dyx \Vert_{F}^2}{20\, \lmax(\oyys)}.
\end{equation*}
In conclusion, combining \eqref{ErrTaylor}, \eqref{InegDerSec} and the upper bounds above,
\begin{eqnarray*}
e_1(\theta, \ts) & \geq & \frac{a_1\, \Vert \dyy \Vert_{F}^2}{8\, \lmax^2(\oyys)} + \frac{a_2\, \lmin(L)\, \Vert \dyx \Vert_{F}^2}{4\, \lmax(\oyys)} + \frac{a_3\, \lmin(\sxxs)\, \Vert \dyx \Vert_{F}^2}{40\, \lmax(\oyys)} \\
 & \geq & \min\left\{ \frac{a_1}{8\, \lmax^2(\oyys)},\, \frac{a_2\, \lmin(L)}{4\, \lmax(\oyys)} + \frac{a_3\, \lmin(\sxxs)}{40\, \lmax(\oyys)} \right\} \Vert \dt \Vert_{F}^2
\end{eqnarray*}
and we clearly identify $\gamma_{r,\eta,\beta,p} > 0$.
\end{proof}

\begin{lem}
\label{LemBorneErr}
Assume that $\lambda$, $\mu$ and $\eta$ are chosen according to the configuration of the theorem. Suppose also that $h_{a}$ in \eqref{MatA} and $h_{b}$ in \eqref{MatB} satisfy
\begin{equation*}
\max\{h_{a}, h_{b}\} < \frac{r^{*}\, \gamma_{r,\eta,\beta,p}}{c_{\lambda, \mu} \sqrt{\vert S \vert}}
\end{equation*}
where $r^{*}$ is given in \eqref{CstR}, $\gamma_{r,\eta,\beta,p}$ in \eqref{CstGam} and $c_{\lambda, \mu}$ in \eqref{CstCb}. Then, under \eqref{H1} and \eqref{H2}, the estimation error satisfies $\Vert \dvt \Vert_{F} \leq r^{*}$.
\end{lem}
\begin{proof}
By convexity of the objective and optimality of $\wht$, each move from $\ts$ in the direction $t\, \dvt$ for $t \in [0, 1]$ must lead to a decrease of the objective, \textit{i.e.}
\begin{equation*}
R_{n}(\ts + t\, \dvt) + \lambda\, \vert \oyys + t\, \dvyy \vert_1^{-} + \mu\, \vert \oyxs + t\, \dvyx \vert_1 - R_{n}(\ts) - \lambda\, \vert \oyys \vert_1^{-} - \mu\, \vert \oyxs \vert_1 \leq 0.
\end{equation*}
Taking the notation of \eqref{Delta}, this can be rewritten as $\Delta_{n}(\ts + t\, \dvt, \ts) \leq 0$. If $\Vert \dvt \Vert_{F} \leq r^{*}$ then choose $t=1$, otherwise calibrate $0 < t < 1$ such that $\Vert t\, \dvt \Vert_{F} = r^{*}$. Then, from Lemma \ref{LemInegSmSp}, it clearly follows that $\ts + t\, \dvt \in N_{r, \alpha}(\ts)$. Hence, the reasoning preceding \eqref{Delta} still holds and, together with Lemma \ref{LemMajErr}, we obtain
\begin{eqnarray*}
0 & \geq & \underline{c}\, \big( \vert [t\, \dvyy]_{\bar{S}} \vert_1 + \vert [t\, \dvyx]_{\bar{S}} \vert_1 \big) - \overline{c}\, \big( \vert [t\, \dvyy]_{S} \vert_1 + \vert [t\, \dvyx]_{S} \vert_1 \big) + \gamma_{r,\eta,\beta,p}\, \Vert t\, \dvt \Vert_{F}^2 \\
 & \geq & -\overline{c}\,\vert [t\, \dvt]_{S} \vert_1 + \gamma_{r,\eta,\beta,p}\, \Vert t\, \dvt \Vert_{F}^2 \\
 & \geq & -c_{\lambda, \mu} \max\{h_{a}, h_{b}\}\, \sqrt{\vert S \vert}\, \Vert t\, \dvt \Vert_{F} + \gamma_{r,\eta,\beta,p}\, \Vert t\, \dvt \Vert_{F}^2 
\end{eqnarray*}
where we used $\underline{c} > 0$ and Cauchy-Schwarz inequality to get $\vert [ \cdot ]_{S} \vert_1^2 \leq \vert S \vert\, \Vert [ \cdot ]_{S} \Vert_{F}^2$. The constant $c_{\lambda, \mu}$ may be explicitely computed from the configuration of $(\lambda, \mu, \eta)$ and is given in \eqref{CstCb}. Note that in the proof of Lemma \ref{LemInegSmSp}, it was sufficient to see that $R_{n}(\theta) - R_{n}(\ts) \geq \phi^{\prime}(0)$ whereas here, we must consider $R_{n}(\theta) - R_{n}(\ts) = \phi^{\prime}(0) + e_1(\theta, \ts)$ to meet our purposes. That explains the presence of $\gamma_{r,\eta,\beta,p}\, \Vert t\, \dvt \Vert_{F}^2$ in the inequality. We deduce that the error must satisfy
\begin{equation*}
\Vert t\, \dvt \Vert_{F} \leq \frac{c_{\lambda, \mu} \sqrt{\vert S \vert} \max\{h_{a}, h_{b}\}}{\gamma_{r,\eta,\beta,p}}.
\end{equation*}
As a corollary, it holds that $\Vert \dvt \Vert_{F} > r^{*}\, \Rightarrow\, c_{\lambda, \mu} \sqrt{\vert S \vert} \max\{h_{a}, h_{b} \} \geq r^{*}\, \gamma_{r,\eta,\beta,p}$ or, conversely written, $c_{\lambda, \mu} \sqrt{\vert S \vert} \max\{h_{a}, h_{b} \} < r^{*}\, \gamma_{r,\eta,\beta,p}\, \Rightarrow\, \Vert \dvt \Vert_{F} \leq r^{*}$.
\end{proof}

\begin{lem}
\label{LemBorneFinale}
Assume that $\lambda$, $\mu$ and $\eta$ are chosen according to the configuration of the theorem. Then, under \eqref{H1}, there exists absolute constants $b_1 > 0$ and $b_2 > 0$ such that, for any $b_3 \in\, ]0,1[$ and as soon as
\begin{equation*}
n \geq \max\big\{ b_1\, (q + \lceil s_{\alpha} \rceil \ln(p+q)),\, \ln(10 (p+q)^2) - \ln(b_3) \big\},
\end{equation*}
with probability no less that $1 - \de^{-b_2 n}-b_3$ both the random hypothesis \eqref{H2} is satisfied and the upper bound
\begin{equation*}
\max\{h_{a}, h_{b} \} \leq 16\, m^{*} \sqrt{\frac{\ln(10 (p+q)^2) - \ln(b_3)}{n}}
\end{equation*}
holds, where $h_{a}$ and $h_{b}$ are given in \eqref{MatA} and \eqref{MatB}, $s_{\alpha}$ is defined in \eqref{CstSa} and $m^{*}$ in \eqref{CstMs}. Hence, one can find a minimal number of observations $n_0$ such that the theorem holds with high probability as soon as $n > n_0$.
\end{lem}
\begin{proof}
All the ingredients of the proof are established in \cite{YuanZhang14}. The authors start by recalling that there exists absolute constants $b_1 > 0$ and $b_2 > 0$ such that hypothesis \eqref{H2} is satisfied with probability no less than $1 - \de^{-b_2 n}$ as soon as $n \geq b_1\, (q + \lceil s_{\alpha} \rceil \ln(p+q))$. We also refer the reader to Lem. 5.1 and Thm. 5.2 of \cite{BaraniukEtAl08}, or to Lem. 7.4 of \cite{Giraud14} for the random bounds of the restricted isometry constants. Afterwards, they prove (see Prop. 4) that, as soon as $n \geq \ln(10 (p+q)^2) - \ln(b_3)$ for some $b_3 > 0$, with probability $1-b_3$,
\begin{equation*}
\max\{h_{a}, h_{b}\} \leq 16\, m^{*} \sqrt{\frac{\ln(10 (p+q)^2) - \ln(b_3)}{n}}.
\end{equation*}
To find the minimal number of observations, we just need to make sure that the above bound is itself smaller than the one of Lemma \ref{LemBorneErr}. It is then not hard to see that we may retain the minimal size $n_0$ given in \eqref{NbMinObs}.
\end{proof}

\appendix
\section{Some constants}
\label{SecApp}
This appendix is entirely dedicated to the constants appearing in the theoretical guarantees. Indeed, a centralization seemed necessary to clarify the rest of the paper, especially the understanding of the main theorem. First, we need to define some constants related to $L$ and to the true values of the model. The bounds
\begin{equation}
\label{BornesVP}
\omli = \frac{\lmin(\oyxs\, L\, \oyxts)}{4\, \lmax(\oyys)}, \hsp \omls = \frac{4\, \lmax(\oyxs\, L\, \oyxts)}{\lmin(\oyys)}, \hsp \omss = \frac{4\, \lmax(\oyxs\, \sxxs\, \oyxts)}{\lmin(\oyys)}.
\end{equation}
are useful to control the eigenvalues of some recurrent expressions (Lemmas \ref{LemBornesVPS} and \ref{LemBornesVPL}), uniformly in a neighborhood of $\ts = (\oyys, \oyxs)$. The true value of the term at the heart of the structural regularization is
\begin{equation}
\label{CstS}
s_{L} = \llangle L, \oyxts\, \oyyis\, \oyxs \rrangle.
\end{equation}
It plays a role in the proof of Lemma \ref{LemInegSmSp} and, as a consequence, in the definition of the area of validity $\Lambda$. This important lemma also  requires to define
\begin{equation}
\label{CstL}
 \hsp \ell_{a} = \vert \oyyis\, \oyxs\, L\, \oyxts\, \oyyis \vert_{\infty} \hsp \text{and} \hsp \ell_{b} = 2\, \vert \oyyis\, \oyxs\, L \vert_{\infty}
\end{equation}
and, in the context of the theorem,
\begin{equation}
\label{CstAlpha}
\alpha = \frac{\max\left\{ \frac{(c_{\lambda}+1) \lambda}{c_{\lambda}} + \eta \beta\, s_{L}^{\, \beta-1} \ell_{a}, \frac{(c_{\mu}+1) \mu}{c_{\mu}} + \eta \beta\, s_{L}^{\, \beta-1} \ell_{b} \right\}}{\min\left\{ \frac{(c_{\lambda}-1) \lambda}{c_{\lambda}} - \eta \beta\, s_{L}^{\, \beta-1} \ell_{a}, \frac{(c_{\mu}-1) \mu}{c_{\mu}} - \eta \beta\, s_{L}^{\, \beta-1} \ell_{b} \right\}}.
\end{equation}
From $\alpha$ and the cardinality of the true active set $\vert S \vert$, let
\begin{equation}
\label{CstSa}
s_{\alpha} = \vert S \vert\, \left[ 1 + \frac{12\, \alpha^2\, \lmax(\sxxs)}{\lmin(\sxxs)} \right]
\end{equation}
which serves as an upper bound in the random hypothesis \eqref{H2}. Similarly, let
\begin{equation}
\label{CstR}
r^{*} = \min\{ r_1^{*}, r_2^{*}, r_3^{*}, r_4^{*} \} 
\end{equation}
where
\begin{equation*}
r_1^{*} = \frac{\lmin(\oyys)}{2}, \hsp r_2^{*} = \frac{\frac{\sqrt{10} - \sqrt{7}}{\sqrt{5}} \sqrt{\lmax(\oyxs\, \sxxs\, \oyxts)}}{\frac{3 \sqrt{3}}{2 \sqrt{2}} \sqrt{\lmax(\sxxs)}}, \hsp r_3^{*} = \frac{\lmin(\oyxs\, L\, \oyxts)}{4\, \Vert L\, \oyxts \Vert_{2}}
\end{equation*}
and
\begin{equation*}
r_4^{*} =  \frac{(\sqrt{2}-1) \sqrt{\lmax(\oyxs\, L\, \oyxts)}}{\sqrt{\lmax(L)}}.
\end{equation*}
Together with $\alpha$ given above, $r^{*}$ is necessary to build the so-called neighborhood $N_{r, \alpha}(\ts)$ defined in \eqref{DefNR}, which plays a fundamental role in all our reasonings. It is important to note that, under the configuration of the theorem and hypothesis \eqref{H1}, $\alpha > 0$ and $r^{*} > 0$. Then, Lemma \ref{LemMajErr} highlights a new constant, characterizing a strong local convexity of the smooth part of the objective in the neighborhood $N_{r, \alpha}(\ts)$,
\begin{equation}
\label{CstGam}
\gamma_{r,\eta,\beta,p} = \min\left\{ \frac{a_1}{8\, \lmax^2(\oyys)},\, \frac{a_2\, \lmin(L)}{4\, \lmax(\oyys)} + \frac{a_3\, \lmin(\sxxs)}{40\, \lmax(\oyys)} \right\}
\end{equation}
where, as it is detailed in the proof of the lemma in question,
\begin{equation*}
a_1 = 1 - \epsilon_{S}\, \omss - \eta \beta\, p^{\, \beta-1}\, \omls^{\, \beta}\, \epsilon_{L}, \hsp a_2 = \frac{2\, \epsilon_{S}}{2+\epsilon_{S}} \hsp \text{and} \hsp a_3 = \eta \beta\, (p\, \omli)^{\beta-1} \frac{2\, \epsilon_{L}}{2+\epsilon_{L}}
\end{equation*}
for some well-chosen $\epsilon_{S} > 0$ and $\epsilon_{L} > 0$. Here again, we make sure that $\gamma_{r,\eta,\beta,p} > 0$. In the same way, in the context of the theorem,
\begin{equation}
\label{CstCb}
c_{\lambda, \mu} = \max\left\{ \frac{(c_{\lambda}+1)\, d_{\lambda}}{c_{\lambda}} + e_{\lambda},\, \frac{(c_{\mu}+1)\, d_{\mu}}{c_{\mu}} + e_{\mu} \right\}
\end{equation}
is needed through Lemma \ref{LemBorneErr}. Finally, independently of the structure matrix $L$,
\begin{equation}
\label{CstMs}
m^{*} = \vert \diag(\sxxs) \vert_{\infty} + \vert \diag(\oyyis\, \oyxs\, \sxxs\, \oyxts\, \oyyis) \vert_{\infty}
\end{equation}
is going to play a significative role in the upper bound of the theorem.

\nocite{*}

\bibliographystyle{acm}
\bibliography{HD_CovEst}

\end{document}